\documentclass[onecolumn,notitlepage,11pt]{article}
\usepackage{sectsty}
\usepackage{latexsym}
\usepackage{amsbsy}
\usepackage{amssymb}
\usepackage{verbatim}
\usepackage{amsthm}
\usepackage{amsfonts}
\usepackage{amsmath}
\usepackage{makeidx}
\usepackage{graphicx}
\usepackage{layout}
\usepackage{pb-diagram}
\usepackage{amscd}
\usepackage{anysize}
\usepackage{tikz}
\usepackage{graphicx}
\usepackage{upgreek}
\usepackage[all,cmtip]{xy}
\usepackage{youngtab}
\usepackage{hyperref}
\usepackage{color}
\newcommand{\refer}[1]{\hyperref[#1]{\ref*{#1}}}%for hyperlinks \label{A}--\refer{A}

\newcommand{\C}{\mathbb{C}}
\newcommand{\p}{\mathcal{P}}

\newcommand{\Z}{\mathbb{Z}}
\newcommand{\Q}{\mathbb{Q}}
\newcommand{\R}{\mathbb{R}}

\newcommand{\T}{\mathbb{T}}

\newcommand{\bra}{\left\langle}
\newcommand{\ket}{\right\rangle}

\newcommand{\Diff}{\mbox{Diff}}
\newcommand{\beq}{\begin{equation*}}
\newcommand{\eeq}{\end{equation*}}
\newtheorem*{thm*}{Theorem}

\newtheorem*{addendum*}{Addendum}
\newtheorem{lemma}{Lemma}[section]

\newtheorem{conjecture}{Conjecture}
\newtheorem*{example*}{Examples}

\newtheorem{rmk}{Remark}
\newtheorem{innercustomthm}{Theorem}
\newenvironment{customthm}[1]
  {\renewcommand\theinnercustomthm{#1}\innercustomthm}
  {\endinnercustomthm}  %%%Use so: \begin{customthm}{YOUR NUMBERING}\
  \newtheorem{innercustomlemma}{Lemma}
\newenvironment{customlemma}[1]
  {\renewcommand\theinnercustomlemma{#1}\innercustomlemma}
  {\endinnercustomlemma}
\newtheorem{innercustomcoro}{Corollary}
\newenvironment{customcoro}[1]
  {\renewcommand\theinnercustomcoro{#1}\innercustomcoro}
  {\endinnercustomcoro}  
\theoremstyle{definition}
\marginsize{3cm}{3cm}{3cm}{3cm}
\subsectionfont{\normalsize}
\sectionfont{\large}
\title{\Large\bf Rigidity and characteristic classes of smooth bundles with
nonpositively curved fibers}
\author{Mauricio Bustamante\and 
F.Thomas Farrell\and 
Yi Jiang}
\newcommand{\Addresses}{{% additional braces for segregating \footnotesize
  \bigskip
  \footnotesize
  \textsc{Mauricio Bustamante}, \textsc{Department of Mathematical Sciences, Binghamton University, Binghamton, NY}\par\nopagebreak
 \texttt{bustamante.math@gmail.com}
  \medskip
  
  \textsc{F.Thomas Farrell}, \textsc{Department of Mathematical Sciences
  and Yau Mathematical Sciences Center, 
  Tsinghua University, Beijing, China}\par\nopagebreak
  \texttt{farrell@math.tsinghua.edu.cn}
  \medskip
  
  \textsc{Yi Jiang}, \textsc{Yau Mathematical Sciences Center, Tsinghua
  University, Beijing, China}\par\nopagebreak
\texttt{jiangyi@amss.ac.cn}
}}
\date{}
\begin{document}
\maketitle

\begin{abstract}
We prove vanishing results for the generalized Miller-Morita-Mumford
classes of some smooth bundles whose fiber is a closed manifold that supports a nonpositively curved Riemannian metric. We also find, under some extra conditions, that the vertical
tangent bundle is topologically rigid.
\end{abstract}
\section{Introduction and statement of results}
In this paper we study the topology of smooth fiber bundles $M\to E\to B$
whose fiber $M$ is a closed manifold that supports a nonpositively curved
Riemannian metric. We obtain two types of results, namely, topological
rigidity for the associated vertical tangent bundle of $E\to B$ and 
vanishing of the tautological classes in positive degrees (also known
as generalized Miller-Morita-Mumford classes).

Rigidity results have appeared in the context of 
negatively curved bundles already, 
i.e. smooth bundles
whose (concrete) fibers can be given 
negatively curved Riemannian metrics varying
continuously from fiber to fiber. In fact, it is proven in
\cite{FarrellGogolev} that the associated vertical \textit{sphere}
bundle of a 
negatively curved bundle is
topologically trivial, provided
the base space $B$ is a simply connected
simplicial complex. Our results, stated below, concern
with the associated vertical \textit{tangent}
bundle of a bundle whose
(abstract) fiber is a nonpositively curved manifold.

With regard to the rational tautological classes,
we show that they vanish
in many instances for bundles with
nonpositively curved fibers, provided
the dimension of the fiber is greater than 2.
This contrasts heavily with the recent results of 
Galatius and Randal-Williams \cite{GRW},
which show that tautological classes \textit{do not} vanish for
some smooth bundles 
with highly connected fibers.

Our results also point in the direction of the following conjecture
by Farrell and Ontaneda:
\begin{conjecture}[Farrell-Ontaneda]
Every negatively curved fiber bundle over a paracompact,
Hausdorff space 
with the homotopy type of a simply
connected finite simplicial complex is topologically trivial.
\end{conjecture}

We now proceed to state the results of this paper.

A \textit{smooth fiber bundle} is a fiber bundle 
$M\to E\xrightarrow{q} B$
whose (\textit{abstract}) fiber $M$
is a smooth finite dimensional manifold and 
whose structure group is
the group $\mbox{Diff}(M)$ of smooth
self-diffeomorphisms of $M$. We denote the
(\textit{concrete}) fiber $q^{-1}(x)$ over $x\in B$ by $M_x$.

Throughout this paper we assume that $M$ is a closed smooth manifold.
We say that a smooth fiber bundle $M\to E\xrightarrow{q} B$ 
is \textit{fiber homotopically trivial} (resp. \textit{topologically trivial})
if there exists a continuous map
$\Phi:E\to M$ such that its restriction to each
fiber $\Phi|_{M_x}:M_x\to M$ is a homotopy equivalence (resp. homeomorphism). Note that in this case the map
$E\to B\times M$ given by $a\mapsto (q(a),\Phi(a))$ 
is a fiber-preserving homotopy equivalence (resp. homeomorphism).

A smooth fiber bundle
$M\to E\to B$ over $B$ is said to have
\textit{nonpositively curved fibers} if
the abstract fiber $M$ supports
a complete Riemannian metric whose sectional curvatures are all nonpositive. 

Given a smooth fiber bundle $M\to E\xrightarrow{q} B$ we define its 
\textit{associated vertical tangent bundle} $\mathfrak{t}E\xrightarrow{\pi} B$
to be the smooth fiber bundle of vectors tangent to the fibers
$M_x$ of $q:E\to B$, that is to say
\beq
\mathfrak{t}E=\bigcup_{x\in B}TM_x,
\eeq
where $TM_x$ is the tangent bundle of $M_x$. 
The projection map $\pi:\mathfrak{t}E\to B$
is nothing but the composite 
$\mathfrak{t}E\xrightarrow{r} E\xrightarrow{q} B$,
where the map $r:\mathfrak{t}E\to E$ sends
a vector in $\mathfrak{t}E$ to its ``foot'' in $E$. 
Thus $\pi:\mathfrak{t}E\to B$ is a smooth fiber
bundle with fiber $TM$. In fact it is an associated bundle to 
$q:E\to B$ in the sense of Steenrod \cite[p. 43]{steenrod}.
Notice that $r:\mathfrak{t}E\to E$ is a vector
bundle of rank equal to the dimension of $M$. 
Also, if $E$ and $B$ are 
smooth manifolds and $\pi$ is a smooth map, $\mathfrak{t}E$
can be identified with the the subbundle of the tangent bundle $TE$
of $E$, consisting of all the vectors $v\in TE$ in the kernel of
the differential of $q$.

Our first theorem is a rigidity result
for the associated vertical tangent bundle 
of a smooth bundle with nonpositively curved fibers.
\begin{customthm}{A}\label{main1}
Let $M^n\to E\to B$ be a smooth fiber bundle with
nonpositively curved $n$-dimensional ($n\neq 3$) closed fiber
over a finite simplicial complex $B$. 
%which has the homotopy type of a finite simplicial complex. 
Assume that the bundle is
fiber homotopically trivial.
Then there exists a continuous map
$\Psi: \mathfrak{t}E\to TM$ such that $\Psi|_{TM_x}:TM_x\to TM$ 
is a homeomorphism for each $x\in B$.
\end{customthm} 

We invite the reader to compare this result with 
\cite[Theorem 1.6]{FarrellGogolev}, where
they find fiber homotopically trivial smooth bundles with 
real hyperbolic fibers, whose associated
\textit{sphere} bundles are \textit{not} topologically trivial.

Using obstruction theory, one can show that 
a smooth fiber bundle $M\to E\to B$ is fiber
homotopically trivial if the fiber $M$ is aspherical, $\pi_1(M)$ is centerless and $B$ is a simply connected finite simplicial complex 
(see for example \cite[Proposition 1.4]{FarrellGogolev}). Hence
Theorem \ref{main1} (cf. Remark \ref{centerless} below) implies the following:

\begin{customcoro}{A.1}\label{simplyconnected}
If $M^n\to E\to B$ is a smooth bundle ($n\neq 3$), with $M$ a closed
negatively curved manifold and $B$ a simply connected finite simplicial complex, then there exists a continuous map
$\Psi: \mathfrak{t}E\to TM$ such that $\Psi|_{TM_x}:TM_x\to TM$ 
is a homeomorphism for each $x\in B$.
\end{customcoro}

We now recall the definition of the tautological classes (also known as
generalized Miller-Morita-Mumford classes)
of an oriented smooth fiber bundle. Let
$M^n\to E\xrightarrow{q} B$ be an \textit{oriented}
 smooth fiber bundle with closed $n$-dimensional
fiber $M$, i.e. $M$ is oriented and the structure group of the bundle is the
group of orientation preserving self-diffeomorphisms of $M$.
Let $\beta: E\to BSO(n)$ be the classifying map for the
vector bundle $\R^n\to \mathfrak{t}E\xrightarrow{r} E$. 
For every cohomology class $c\in H^i(BSO(n);\Z)$, we define
a \textit{tautological class} $\tau_c(E)$ for the smooth bundle 
$M\to E\xrightarrow{q} B$ by the formula
\beq
\tau_c(E):=q_!\, \beta^*(c)\in H^{i-n}(B;\Z),
\eeq
where
$q_!:H^*(E;\Z)\to H^{*-n}(B;\Z)$ is the Gysin map 
(``integration along the fiber''), arising from the Serre
spectral sequence for the fiber bundle $E\to B$. 
(See \cite[pag. 148-150]{moritabook}.)
We  say that $\tau_c(E)\in H^{i-n}(B;\Z)$ is a class of
\textit{positive degree}  (resp. \textit{degree zero}) if 
$i>n$ (resp. $i=n$).

Our next result 
is another instance of rigidity for the
associated vertical tangent bundle. It
has the consequence that the rational tautological
characteristic classes of a smooth bundle with nonpositively
curved fiber depend only on the characteristic classes of the
abstract fiber.

Recall that a \textit{topological $\R^n$-bundle} is a 
fiber bundle whose fiber is
homeomorphic to $\R^n$. Its structure group is
denoted by $\mbox{TOP}(n)$ and consists of all self-homeomorphisms of $\R^n$. 
%The
%subgroup of all orientation preserving self-homeomorphisms of $\R^n$ will
%be denoted by $\mbox{STOP}(n)$. 
Additionally we can assume that every topological
$\R^n$-bundle has a zero section. This is possible because
the map $p:\mbox{TOP}(n)\to \R^n$, defined by $p(f)=f(0)$ is a fibration whose
base space is contractible. Then $\mbox{TOP}(n)$ is homotopy equivalent to
the fiber of $p$, namely
the group of self-homeomorphisms that fix the origin. The latter
corresponds to the structure group for topological $\R^n$-bundles with a zero-section.

Let $\rho: B\times M\to M$ be the 
projection map onto the second factor.
\begin{customthm}{B}\label{tangential}
Let $M^n\to E\to B$ be a smooth fiber bundle over a 
finite simplicial complex $B$, with fiber a closed
nonpositively curved $n$-dimensional manifold ($n\neq 3$).
Assume that the bundle is
fiber homotopically trivial and let $\Phi:E\to B\times M$ be a fiber
homotopy equivalence.
Then $\mathfrak{t}E$ and the pullback of
the tangent bundle of $M$ along
the composition $E\xrightarrow{\Phi}B\times M\xrightarrow{\rho} M$
are isomorphic as topological $\R^n$-bundles.
\end{customthm}
\begin{rmk}
Although Theorems \ref{main1} and \ref{tangential} appear to be
essentially the same, they are not. In fact in Theorem
\ref{main1} we obtain a homeomorphism which does
not necessarily cover the given fiber
homotopy equivalence between the smooth bundles. Thus Theorem
\ref{tangential} is not a consequence of Theorem
\ref{main1}, and viceversa.
\end{rmk}
As a consequence of Theorem \ref{tangential} we have:
\begin{customcoro}{B.1}\label{main2}
Let $M^n\to E\to B$ be a smooth fiber bundle as in the statement of Theorem
\ref{tangential} (without excluding $n=3$). 
Assume that the bundle is oriented. Then
for all $i>n$ and for any rational cohomology class $c\in H^i(BSO(n);\Q)$,
\begin{equation}\label{eq1}
\tau_{c}(E) = 0.
\end{equation}
\end{customcoro}
\begin{rmk}
If one assumes that the fiber in Corollary \ref{main2} is
negatively curved, it is
possible to drop the condition of fiber 
homotopy triviality and still
get the vanishing of the rational tautological 
classes in positive degrees 
(cf. Theorem \ref{app1} and Example a) below).
\end{rmk}
\begin{rmk}
Tautological classes of degree zero of
(not necessarily fiber homotopically trivial) oriented 
smooth $M^n$-bundles
$q:E\to B$ over a path-connected space $B$ satisfy
the following equation:
\begin{equation}\label{eq2}
\bra\tau_{c}(E),[1_B]\ket=\bra\alpha^*(c),[M]\ket,\ \ \ \text{if}\ \  c\in H^n(BSO(n);\Q),
\end{equation}
where 
$\alpha:M\to BSO(n)$ denotes the classifying map
 for the tangent bundle
of $M$ and $[1_B]\in H_0(B;\Q)$ denotes the image of 
the generator of $H_0(B;\Z)$ and $[M]\in H_n(M;\Q)$ the fundamental
class of $M$. 

To prove this equation, identify the abstract fiber $M$ with some
concrete fiber $M_{\ast}$, $\ast\in B$. Now, if
$\iota:M\hookrightarrow E$ denotes the inclusion map we
obtain a morphism of smooth $M$-bundles
\beq
\xymatrix{
M\ar[r]^{\iota}\ar[d]_{h} & E\ar[d]^{q}\\
\ast\ar[r]^{\subset}&B
}
\eeq
Thus equation (\ref{eq2}) follows straightforwardly from
the naturality of the tautological classes.
%Then the classifying map $\alpha$ is $\beta\circ\iota$, where
%$\beta:E\to BSO(n)$ is the classifying map for the associated
%vertical fiber bundle $\mathfrak{t}E\to E$ and 

\end{rmk}
One can actually realize many non-vanishing classes (in degree zero)
appearing in the
previous remark by bundles with nonpositively curved fibers, indeed
with negatively curved fibers. Here is an example: let
$c\in H^{4k}(BSO(4k);\Q)$
be a monomial of degree $4k$ in the rational Pontrjagin classes
$p_1,\ldots, p_{k}$ such that the corresponding 
Pontrjagin number of the 
complex projective space $\C P^{2k}$ 
is non-zero \cite[p.194]{milnor}.
By Ontaneda's work on smooth Riemannian
hyperbolization \cite{pedro}, the 
cobordism class of $\C P^{2k}$ can
be represented by a closed negatively
curved smooth $4k$-manifold $M$. Hence
the corresponding tautological class $\tau_c(M)$ of the smooth bundle $M\to M\to\ast$ over a point does not vanish, in fact it equals its corresponding Pontrjagin number by the equation (\ref{eq2}). 
%\begin{coro}\label{main3}
%Let $M^n\to E\to B$ be as in the statement of Theorem \ref{tangential}.
%Assume that $M$ and $B$ are oriented manifolds.
%If $n=4k$ or $n=4k+2$ then the tautolgical classes of $E\to B$ depend only
%on the oriented cobordism class of $M$ and the Euler class of the tangent 
%bundle of $M$.
%\end{coro}
%We can isolate an interesting particular case out of Theorem \refer{main2}:
%\begin{coro}\label{torus}
%The tautological classes of positive degree of any smooth torus bundle
%$T^n\to E\to B$ over a $2$-connected finite simplicial complex $B$ vanish.
%\end{coro}

A modification of the proofs of 
Theorem \ref{tangential} and Corollary \ref{main2} yields the
following theorem and its corollary. Recall that
a smooth bundle $M\to E\to B$ is a nonpositively curved bundle if each
concrete fiber $M_x$ can be endowed with a nonpositively curved 
metric $g_x$ in a way that the metrics $g_x$ vary continuously from
fiber to fiber (c.f. \cite{gafa}).

\textbf{Caveat}. Not every smooth bundle with nonpositively curved fibers
is a nonpositively curved bundle. In fact, for every closed negatively
curved manifold $M$ of dimension greater than $9$, it is possible
to construct smooth $M$-bundles over a circle which cannot be given
a fiberwise 
nonpositively curved metric \cite{FO15}.
\begin{customthm}{C}\label{add}
Let $E\to B$ and $E'\to B$ be fiber homotopy equivalent 
smooth $M^n$-bundles over a finite simplicial complex $B$,
where $M^n$ is a closed smooth manifold of 
dimension $n\neq 3$. Assume that 
$M\to E\to B$ is a nonpositively curved bundle.
%the associated vertical tangent bundles 
%$\mathfrak{t}E\to B$ and $\mathfrak{t}E'\to B$
%are topologically equivalent, i.e. there is a fiber-preserving 
%homeomorphism between the bundles $\mathfrak{t}E\to B$ and 
%$\mathfrak{t}E'\to B$.
If $\Phi:E\to E'$ is a fiber homotopy equivalence then
$\Phi^*(\mathfrak{t}E')$ and $\mathfrak{t}E$ are isomorphic as
topological
$\R^n$-bundles over $E$. 
\end{customthm}
\begin{customcoro}{C.1}\label{sameclasses}
Under the hypothesis of the Theorem \ref{add}
(including n=3), the bundles 
$E\to B$ and $E'\to B$ have the same rational tautological 
characteristic classes, provided they are oriented and the fiber
homotopy equivalence $\Phi$ preserves the orientation.
\end{customcoro}
%It might be interesting to look at examples like $T^n\to E\to B$ where 
%$B=SL_n(\Z)\backslash SL_n(\R)/SO(n)$ and $\pi_1B=SL_n(Z)$ acts on $T^n$ via 
%$T^n=\R^n/\Z^n$ and $Aut(\pi_1T^n)=GL_n(\Z)$. More specifically, analyze the bundle\\
%$T^n\to T^n\times_{SL_n\Z}SL_n(\R)/SO(n)\to SL_n(\Z)\backslash SL_n(\R)/SO(n)$.

This invariance of the characteristic classes under fiber 
homotopy equivalence can
be used to obtain vanishing results for characteristic classes
in concrete cases where the fiber is a torus or 
a closed hyperbolic manifold, without the assumption of
fiber homotopy triviality.

\begin{customthm}{D}\label{torus}
Let $\T^n\to E\to B$ be a smooth fiber bundle over a compact smooth manifold $B$ with fiber an $n$-dimensional
torus $\T^n$. Then the rational Pontrjagin
classes of its associated vector bundle $r:\mathfrak{t}E\to E$ over
$E$, vanish.
\end{customthm}
Recall that every rational cohomology class $ c\in H^*(BSO(n);\Q)$
can be expressed as a polynomial in the Pontrjagin classes and
the Euler class, such that the Euler class in each monomoial
has degree at most 1. Also, equation (\ref{eq2}) and
the fact that the Euler characteristic of a torus is 
identically zero, imply that
the tautological class $\tau_e(E)$ 
corresponding to the Euler class 
$e\in H^n(BSO(n);\Q)$ of an oriented
torus bundle $\T^n\to E\to B$ vanishes. Hence
we obtain the following corollary:
\begin{customcoro}{D.1}\label{vanishingtorus}
Let  $\T^n\to E\to B$ be an oriented torus bundle
over a compact smooth
manifold $B$. Then $\tau_c(E)=0$ 
for all $c\in H^*(BSO(n);\Q)$.
\end{customcoro}

\begin{customthm}{E}\label{hyperbolic}
Let $M\to E\to B$ be an oriented smooth fiber 
bundle over a finite simplicial complex $B$, with fiber an
$n$-dimensional ($n\geq 3$) (real, complex or quaternionic) hyperbolic manifold.
Then $\tau_c(E)=0$ 
for all $c\in H^i(BSO(n);\Q)$ with $i>n$.
\end{customthm}
%%%%%%%%%%%%%%%%%%%%%%%%%%%%%%%%%%%%%%%%%%%%%%%%%%%%%%%%55
%%%%%%%%%%%%%%%%%%%%%%%%%%%%%%%%%%%%%%%%%%%%%%%%%%%%%%%%%%
Theorem \ref{hyperbolic} also follows from the 
next more general result, which is proved by different 
methods in the Appendix.
\begin{customthm}{F}\label{app1}
Let $M\to E\xrightarrow{q} B$ be an oriented smooth $M$-bundle 
over a finite simplicial complex $B$. Assume that $M^n$ ($n\geq 3$)
is a nonpositively curved closed manifold such that
$\mbox{Out}(\pi_1M)$ is finite and $\pi_1M$ is centerless.
Then $\tau_c(E)=0$ for all $c\in H^i(BSO(n);\Q)$ with $i>n$.
\end{customthm}
%%%%%%%%%%%%%%%%%%%%%%%%%%%%%%%%%%%%%%%%%%%%%%%%%%%%%%%%%%%%%%
\begin{example*}
\ 
\begin{itemize}
\item[a)] Theorem \ref{app1} holds when $M$ is a closed negatively curved 
Riemannian manifold of dimension at least 3 
(\cite[Theorem 5.4.A]{gromov}, cf. Remark \ref{centerless}).
\item [b)] Theorem \ref{app1} also holds if $M$ is a nonpositively curved locally symmetric space
of non-compact type, such that
 it has no finite sheeted cover $\widehat{M}$
such that $\widehat{M}$ is a metric product $A\times B$ of Riemannian
manifolds where $A$ is 2-dimensional. Notice that in this case
$\mbox{Out}(\pi_1M)$ is finite by Mostow's strong rigidity theorem
\cite[Theorem 24.1]{mostow},
and the fundamental group is centerless as shown for example in
\cite[p.210]{eberlein83}.
\item[c)] The conclusion of Theorem \ref{app1} fails if one allows
the outer automorphisms group of the fundamental group of the fiber to be
infinite, yet the fundamental group itself can be centerless. For an example, consider
the product $E\times E$ of an oriented surface
bundle $N\to E\xrightarrow{q} B$ over a closed surface B  with itself. Denote by $\tau_{\cal{L}}(E\times E)$ 
the tautological class corresponding to the $\cal{L}$-genus
of $\mathfrak{t}(E\times E)\to E\times E$, i.e. 
\beq
\tau_{\cal{L}}(E\times E)=
(q\times q)_!\left(\frac{7}{45}p_2(\mathfrak{t}(E\times E))+
	\frac{1}{45}p_1^2(\mathfrak{t}(E\times E))\right)\in 
	H^4(B\times B;\Q).
\eeq
It is not difficult to conclude that
\beq
\left<\tau_{\cal{L}}(E\times E),[B\times B]\right>=
\mbox{sign}\, (E\times E)
=(\mbox{sign}\, E)^2.
\eeq
And it is known (see \cite{atiyahsign} or \cite[p.155-160]{moritabook}) that 
there exist surface bundles over surfaces for which $\mbox{sign}\, E\neq 0$.

Note that the example $b)$ does not apply to the
bundle $E\times E\to B\times B$ since $N\times N$
obviously has a 2-dimensional metric factor.
On the other hand, there are closed 
irreducible nonpositively curved $4$-dimensional
Riemannian manifolds $M$ (i.e. no finite sheeted cover of $M$ is
a nontrivial metric product) 
which are locally isometric to $N\times N$
(i.e. the universal cover of $M$ is isometric to the product of 
hyperbolic spaces 
$\mathbb{H}\times\mathbb{H}$) and hence 
satisfy the conditions of example
b). Therefore the fundamental group $\pi_1M$ is centerless and 
$\mbox{Out}(\pi_1M)$ is finite. Consequently
Theorem \ref{app1} applies
to oriented smooth bundles with fiber $M$. The existence of such manifolds $N$ is a consequence of 
\cite[Theorem C]{johnson} and \cite[\S 6]{shimizu}. 
\end{itemize}
\end{example*}
%%%%%%%%%%%%%%%%%%%%%%%%%%%%%%%%%%%%%%%%%%%%%%%%%%%%%%%%
Under the additional assumption that
the fiber $M$ of a smooth bundle
satisfies the Strong Borel Conjecture, we can show 
the invariance of the tautological classes 
under fiber homotopy equivalence of $M$-bundles.
We say that a closed manifold $M$ 
\textit{satisfies the Strong Borel
Conjecture (SBC)} if, for 
all $k\geq 0$,
every self-homotopy equivalence of pairs 
$(M\times D^k,M\times S^{k-1})\to (M\times D^k,M\times S^{k-1})$ which
is a homeomorphism when restricted to the boundary $M\times S^{k-1}$ is
homotopic (relative to the boundary) to a homeomorphism, where
$D^k$ is a $k$-dimensional closed disc. For example, every closed
nonpositively curved Riemannian manifold of dimension $\geq 5$ 
satisfies SBC, since it is topologically rigid  \cite{farrelljones2}. (See
\cite{farrelljones1} for more results about topological rigidity).
%%%%%%%%%%%%%%%%%%%%%%%%%%%%%%%%%%%%%%%%%%%%%%%%%%%%%%%%%%%
%%%%%%%%%%%%%%%%%%%%%%%%%%%%%%%%%%%%%%%%%%%%%%%%%%%%%%%%%%
\begin{customthm}{G}\label{app3}
Let $\Phi:E\to E'$ be a fiber homotopy equivalence between 
$M$-bundles $p:E\to B$ and $q:E'\to B$
over a finite simplicial complex $B$. Assume that $M$ is a
closed smooth manifold which satisfies SBC. Then 
$\Phi^*(p_i(\mathfrak{t}E'))=p_i(\mathfrak{t}E)$, where 
$p_i(\mathfrak{t}E)\in H^{4i}(E;\Q)$ denotes the $i$-th rational
Pontrjagin class of $\mathfrak{t}E\to E$.
\end{customthm}

\begin{customcoro}{G.1}\label{app4}
Let  $E\to B$ and $E'\to B$ be oriented smooth 
$M$-bundles over a finite simplicial complex $B$. 
Assume that $M$ satisfies SBC. If there
is an orientation preserving fiber
homotopy equivalence $E\to E'$, then
\beq
\tau_c(E)=\tau_c(E')
\eeq
for all $c\in H^*(BSO(n);\Q)$.
\end{customcoro}
\begin{rmk}
With the exception of Theorem \ref{torus} and
Corollary \ref{vanishingtorus}, all the other
results stated above can be 
proven under the more general 
assumption that the base space $B$ is a paracompact, Hausdorff 
space which is homotopy equivalent to a finite
simplicial complex $\mathcal{B}$ 
(for example, $B$ could be a compact topological manifold \cite{triangulation}). In that case the 
proofs follow from rather direct arguments involving 
pull-backs of bundles along some homotopy
equivalence $\mathcal{B}\to B$ and a homotopy 
inverse.
We do not state them in that generality in order not to obscure the main ideas.
\end{rmk}
This article is organized as follows: In Section \ref{topotri} we 
prove Theorem \ref{main1}. In Section \ref{vanishing} we prove \ref{tangential} and Corollary \ref{main2}. Theorem \ref{add} and 
Corollary \ref{sameclasses} are proven in Section \ref{nonpositively}.
In Section \ref{torusbundle}
we prove Theorem \ref{torus}, Corollary
\ref{vanishingtorus} and Theorem \ref{hyperbolic}.
Since the proofs of Theorems \ref{app1} and
\ref{app3} and Corollary \ref{app4} follow a different strategy than that used to prove
the other results, we do 
them in the
Appendix at the end of this article. In the same appendix we 
prove a lemma about the invariance of the Euler class under 
fiber homotopy equivalence 
(Lemma \ref{app2}, see Section \ref{nonpositively}). 

\subsection*{Acknowledgements}
The first author is grateful to Tom Farrell and Pedro Ontaneda
for the
financial support from their NSF grant
during the year when this project
was carried out. He also acknowledges the hospitality
of the Yau Mathematical Sciences Center of Tsinghua University
in Beijing. The second author was partially supported by NSF
grant DMS-1206622. We are grateful to Oscar Randal-Williams for
his comments on an earlier
version of this paper, which led to a correct
proof of Lemma \ref{app2}.
We acknowledge the referee, who provided us with
very valuable comments and suggestions that made the proofs of
some of our theorems much more clear, particularly those in
Sections \ref{topotri} and \ref{nonpositively}.
%%%%%%%%%%%%%%%%%%%%%%%%%%%%%%%%%%%%%%%%%%%%%%%%%%%%%%%%%%%%%%%
%%%%%%%%%%%%%%%%%%%%%%%%%%%%%%%%%%%%%%%%%%%%%%%%%%%%%%%%%%%%%%%
\section{Topological type of the vertical tangent bundle}\label{topotri}
In this section we prove Theorem \refer{main1}. We handle the cases
$n\geq 4$, $n=2$ and $n=1$ separately. 
We will make use of the following:
\begin{lemma}{\cite[Corollary 2.7]{FW91}}\label{coro2.7}
Let $n\geq 4$ and $k$ be given. Then there is a $\delta>0$ such that if
$X^k$ is a $k$-dimensional complex, $p:E\to X^k$ is a vector bundle with fiber
$\R^n$, $p':E'\to X'$ is a topological $\R^n$-bundle, and $\hat{f}:E\to E'$
is a bundle map that covers $f:X\to X'$ and such that the restriction of $\hat{f}$ to each fiber is a
$\delta$-map, then $f^*(E')$ is isomorphic to $E$ as a topological 
$\R^n$-bundle.
\end{lemma}
A map $g:Y\to Z$ between two metric spaces $Y$ and $Z$
is a $\delta$-\textit{map} if $\mbox{diam}(g^{-1}(z))<\delta$, for each $z\in Z$.

Here and in the next sections we will need the following
construction: Let $\Gamma$ be a group. Recall that if $X$ is a 
right $\Gamma$-space and $Y$ is a left $\Gamma$-space, then
we can define $X\times_{\Gamma}Y$ to be the quotient space $X\times Y/\sim$, where the equivalence relation on $X\times Y$ is given by
$(x,y)\sim (x',y')$ if there exists $g\in\Gamma$
such that $x'=xg^{-1}$ and $y'=gy$.

Let $\Gamma\to E\Gamma\xrightarrow{\pi} B\Gamma$ be the
universal principal $\Gamma$-bundle over $B\Gamma$ and $F$ be a left $\Gamma$-space. The 
\textit{associated} fiber bundle over $B\Gamma$ with fiber $F$ is denoted
by 
\beq
F\to E\Gamma\times_{\Gamma}F\to B\Gamma,
\eeq
where the bundle projection sends a class
$[e,x]\in E\Gamma\times_{\Gamma}F$ to $\pi(e)$.

\begin{proof}[\textbf{Proof of Theorem \refer{main1}}]\ \\
\textbf{Case $n\geq 4$:} Our proof in this case is essentially
a parametrized version of the proof of a theorem of Ferry and Weinberger
\cite[Theorem 2]{FW91}. Let $\Phi:E\to B\times M$ be a homotopy trivialization of the
smooth bundle $M\to E\xrightarrow{q} B$. 
Identify $M$ with the concrete fiber $M_{\ast}=q^{-1}(\ast)$, 
$\ast\in B$, and endow it with a nonpositively curved Riemannian metric.

Fix a universal covering $\rho_M:\widetilde{M}\to M$ and 
denote by $\Gamma$ its
group of deck transformations. Let $\widetilde{E}$ be the pullback
of the covering 
$B\times\widetilde{M}\xrightarrow{id_B\times\rho_M}B\times M$ 
along $\Phi:E\to B\times M$ and 
$\widetilde{\Phi}:\widetilde{E}\to B\times\widetilde{M}$ the map
that covers $\Phi$. 

Consider now the 
following diagram:
\begin{equation}\label{diagram1}
\xymatrixcolsep{4pc}\xymatrix{
        B\times TM\ar[r]^-{id_B\times Exp}\ar[dr]_{id_B\times\tau_M}    
     & B\times\widetilde{M}\times_{\Gamma}\widetilde{M}\ar[d]_{id_B\times u}     
     \ar[r]^-{\widetilde{\Upsilon}\times_{\Gamma} id}      
     &\widetilde{E}\times_{\Gamma}\widetilde{M}\ar[dl]^-{(q\circ\rho_E)
     \times_{\Gamma}\rho_M}\\
     &B\times M&
}
\end{equation}
where the maps are defined as follows:
\begin{itemize}
\item $\tau_M:TM\to M$ is the tangent bundle projection.
\item To define
$Exp:TM\to \widetilde{M}\times_{\Gamma}\widetilde{M}$,
first consider the map $T\widetilde{M}\to\widetilde{M}\times\widetilde{M}$ 
given by $v\mapsto (\gamma_v(1),\gamma_v(0))$, where $\gamma_v$ is the unique
geodesic in $\widetilde{M}$ such that $\gamma_v(0)=\tau_{\widetilde{M}}(v)$
and $\dot{\gamma}_v(0)=v$. This map is $\Gamma$-equivariant respect to
the action of $\Gamma$ on $T\widetilde{M}$ induced by the $\Gamma$-action
on $\widetilde{M}$. The map
$Exp:TM\to\widetilde{M}\times_{\Gamma}\widetilde{M}$ is the induced map on
the orbit space.
\item $u:\widetilde{M}\times_{\Gamma}\widetilde{M}\to M$
is the projection onto the second
factor mod $\Gamma$.
\item The map $\widetilde{\Upsilon}\times_{\Gamma}id$ is defined as 
follows:
Let $\Upsilon:B\times M\to E$ be a (fiber preserving) 
homotopy inverse of $\Phi$ (cf. \cite{dold2}). Thus 
$\Upsilon^*\widetilde{E}$ is isomorphic to $B\times\widetilde{M}$.
Consequently we obtain an
isomorphism of $\Gamma$-coverings which in turn gives rise to a lifting
$\widetilde{\Upsilon}:B\times\widetilde{M}\to \widetilde{E}$ of
$\Upsilon$. The map 
$\widetilde{\Upsilon}\times_{\Gamma}id$ corresponds to the map induced by $\widetilde{\Upsilon}\times id$ on the $\Gamma$-orbit space.
\item $\rho_E:\widetilde{E}\to E$ is the natural projection.
\end{itemize}

Notice that the left diagonal map is a vector bundle and the right diagonal
map is a topological $\R^n$-bundle. 
One easily checks that every triangle in the  
\hyperref[diagram1]{diagram \ref*{diagram1} above} commutes. 

Hence
$(\widetilde{\Upsilon}\times_{\Gamma} id)\circ (id_B\times Exp):B\times TM\to \widetilde{E}\times_{\Gamma}
\widetilde{M}$ is a fiber preserving map. Moreover, for each $x\in B$,
the restriction 
$\widetilde{\Upsilon}|_{\{x\}\times\widetilde{M}}
:\{x\}\times\widetilde{M}\to\widetilde{M_x}$ 
is a $\delta_x$-map for some
$\delta_x\geq 0$. Thus, since $B$ is compact, 
the restriction of $\widetilde{\Upsilon}$ to each fiber
is a $\delta$-map for any 
$\delta\geq\max\{\delta_x\}$.
Since we are choosing a
nonpositively curved metric on $M$,
the exponential map $Exp$ is a weakly expanding map
by the Cartan-Hadamard theorem, i.e.
for each $p\in M$, and $v_1,v_2\in T_pM$
\beq
d(Exp(v_1),Exp(v_2))\geq d(v_1,v_2).
\eeq
Hence (using the linear structure on the fibers
of $TM$ to obtain a suitable $\delta$ if necessary)
the map 
$(\widetilde{\Upsilon}\times_{\Gamma} id)\circ (id_B\times Exp):
B\times TM\to\widetilde{E}\times_{\Gamma}
\widetilde{M}$
satisfies the conditions of Lemma \ref{coro2.7}, and this implies that
there exists an $\R^n$-bundle isomorphism
\beq
h:B\times TM\to\widetilde{E}\times_{\Gamma}\widetilde{M}.
\eeq

We now consider the $\R^n$-bundle $\tau E\xrightarrow{t}E$,
which is the pullback of the natural projection
$\rho_E\times_{\Gamma}(q\circ\rho_E):
\widetilde{E}\times_{\Gamma}\widetilde{E}\to
E	\times B$
along the map $(id_E,q):E\to E\times B$

Another $n$-vector bundle $V$	 over $E$ can be obtained as follows: Denote by 
$pr_{\widetilde{M}}:B\times\widetilde{M}\to\widetilde{M}$
the projection onto the second factor and let
$s:E\to\widetilde{E}\times_{\Gamma}\widetilde{M}$ be
the section induced by
$(id_{\widetilde{E}},pr_{\widetilde{M}}\circ\widetilde{\Phi})
:\widetilde{E}\to\widetilde{E}\times\widetilde{M}$ 
on orbit spaces . The tangent bundle
projection $T\widetilde{M}\to\widetilde{M}$ induces a 
vector bundle 
$\widetilde{E}\times_{\Gamma}T\widetilde{M}\to\widetilde{E}\times_{\Gamma}\widetilde{M}$.
Let $V\to E$ be the pullback of this vector bundle along 
$s$. By composing the
map $V\to\widetilde{E}\times_{\Gamma}T\widetilde{M}$
that covers $s:E\to\widetilde{E}\times_{\Gamma}\widetilde{M}$
with the natural map
$\widetilde{E}\times_{\Gamma}T\widetilde{M}\to\widetilde{E}\times_{\Gamma}\widetilde{M}$
induced by the exponential map (defined with respect to the complete
nonpositively curved metric on $\widetilde{M}$), we 
obtain a topological $\R^n$-bundle homeomorphism 
denoted by
$\omega:V\to\widetilde{E}\times_{\Gamma}\widetilde{M}$.

In addition, note that the map
\beq
\widetilde{E}\times_{\Gamma}\widetilde{M}\to
\widetilde{E}\times_{\Gamma}\widetilde{E},
\eeq 
that sends a class $[e,y]\in\widetilde{E}\times_{\Gamma}\widetilde{M}$
to $[e,\widetilde{\Upsilon}(q(\rho_E(e)),y)]$,
covers the map $(id_E,q):E\to E\times B$.
Hence there is a map 
$\lambda:\widetilde{E}\times_{\Gamma}\widetilde{M}\to\tau E$ that
makes the following diagram commutative:
\begin{equation}\label{diagram2}
\xymatrixcolsep{4pc}\xymatrix{
        V\ar[r]^{\omega\ \ \ }\ar[dr]    
     & \widetilde{E}\times_{\Gamma}\widetilde{M}\ar[d]_{\rho_E}     
     \ar[r]^{\ \ \ \lambda\ }      
     &\tau E\ar[dl]^{\ \ t}\\
     &E&
}
\end{equation}

Arguing as a above, we find that the composition 
$\lambda\circ\omega$ in
\hyperref[diagram2]{diagram \ref*{diagram2}} is a bundle map whose restriction
to each fiber is a $\delta$-map.
We invoke Lemma \ref{coro2.7} again to conclude that there is a fiber-preserving
homeomorphism $h':V\to\tau E$.

Note that the restriction of the bundle map 
$\tau E\xrightarrow{(h'\circ\omega^{-1}\circ h)^{-1}} B\times TM\xrightarrow{pr} TM$ to each fiber over $x\in B$ is 
a homeomorphism. Thus the proof of Theorem \ref{main1} in the
case $n\geq 4$ will be completed as soon as
we prove the following lemma:
\begin{lemma}\label{isoasbundles}
$\mathfrak{t}E$ and $\tau E$ are isomorphic topological 
$\R^n$-bundles over $E$.
\end{lemma}
\begin{proof}
We prove this lemma by first showing that $\mathfrak{t}E$ and $\tau E$
are isomorphic as microbundles. (We refer the reader to
\cite{micromilnor} for the basic
material about microbundles.) To see this, first choose 
a continuously varying family of Riemannian
metrics $g_x$, $x\in B$, on each fiber $M_x$. 
Let
$exp^{\, g_x}:T\widetilde{M}_x\to\widetilde{M_x}$ 
be the exponential map associated to the Riemannian metric on
$\widetilde{M}_x$ induced by $g_x$. 
Define $\mathfrak{t}\widetilde{E}\to\widetilde{E}\times\widetilde{E}$
by sending $v\in \mathfrak{t}\widetilde{E}$ to 
$[\widetilde{r}(v),exp^{g_x}(v)]$, where 
$\widetilde{r}:\mathfrak{t}\widetilde{E}\to\widetilde{E}$ is the
bundle projection and $x=q(\rho_E(\widetilde{r}(v)))$.
Since the 
metrics vary continuously from fiber to fiber, we obtain a
$\Gamma$-equivariant map 
$\mathfrak{t}\widetilde{E}\to\widetilde{E}\times\widetilde{E}$,
and moding out by $\Gamma$ we have
the following commutative diagram:
\beq
\xymatrix{
\mathfrak{t}E\ar[r]\ar[d] & \widetilde{E}\times_{\Gamma}\widetilde{E}\ar[d]\\
 E\ar[r]^-{(1,q)} & E\times B
}
\eeq
where the vertical maps are the obvious bundle projections.
This diagram, in turn,  gives rise to a map
\beq
\zeta:\mathfrak{t}E\to\tau E=(1,q)^*(\widetilde{E}\times_{\Gamma}\widetilde{E})
\eeq

As $g_x$ depends
continuously on $x$, the function on $M_x$
that assigns to each point 
in $M_x$ its injectivity
radius, varies continuously over $x\in B$.
Since $B$ and $M_x$ are both
compact, let $r>0$ be the minimum among all these injectivity radii.
Denote by $N_r(M_x)$ an $r$-neighborhood of 
$M_x$ in $TM_x$. Then the restriction of
the map $\zeta$ to 
the neighborhood $N_r(E)=\bigcup_{x\in B} N_r(M_x)$ of 
the zero section $E$ in $\mathfrak{t}E$, 
is an embedding into a neighborhood of the zero section in $\tau E$.
This proves that $\mathfrak{t}E$ and $\tau E$ are isomorphic
microbundles. Hence, by Kister-Mazur's theorem
\cite[Theorem 2]{kister}, $\mathfrak{t}E$ and $\tau E$ must be 
isomorphic as topological $\R^n$-bundles. 
This completes the proof of Lemma 
\ref{isoasbundles} and Theorem \ref{main1} when $n\geq 4$.
\end{proof}

We now prove Theorem \ref{main1} in the remaining cases.
Let $B\mbox{Diff}(M)$ be the classifying space for the group
$\mbox{Diff}(M)$ of 
self-diffeomorphisms of $M$  and let $BG(M)$ 
be the classifying space for 
the associative $H$-space $G(M)$ of self-homotopy equivalences of $M$
(see \cite{dold}). Define $\sigma:B\mbox{Diff}(M)\to BG(M)$ to be the map
functorially induced by the inclusion (which
is an $H$-space morphism) $\iota:\mbox{Diff}(M)\hookrightarrow G(M)$.
It follows from \cite{dold} that two smooth bundles
$f_1,f_2:B\to B\mbox{Diff}(M)$ are fiber homotopy equivalent if and only if
$\sigma\circ f_1$ is homotopic to $\sigma\circ f_2$. Thus, if $\sigma$
is a weak homotopy equivalence, then
every fiber homotopically trivial
smooth $M$-bundle over $B$ is also smoothly trivial (see 
\cite[Proposition 4.22]{hatcher}). This implies 
that the associated vertical
tangent bundle will be topologically (in fact smoothly) trivial.

Our strategy to prove Theorem \ref{main1} in lower dimensions consists of 
proving that the map $\sigma$ is a weak homotopy equivalence. For this
it suffices to show that $\iota:\mbox{Diff}(M)\to G(M)$ 
is a weak homotopy equivalence. 

Recall also that the space $G(X)$ of self-homotopy equivalences of any
aspherical complex $X$ has the homotopy type of
$\mbox{Out}(\pi_1X)\times K(\mbox{Center}(\pi_1X),1)$ 
(see \cite{G(X)}). It will also be
useful to keep in mind that $\mbox{Diff}(M)=\pi_0\mbox{Diff}(M)\times \mbox{Diff}^0(M)$, where 
$\mbox{Diff}^0(M)$ denotes the connected component of the identity.

With all this preparation, we can prove the theorem in dimensions $2$ and $1$.

\textbf{Case $n=2$:}
The closed surfaces that support a nonpositively curved metric are the
orientable surfaces of genus $\geq 1$ and non-orientable surfaces of genus
$g\geq 2$. 

It follows from classical theorems of Dehn, Nielsen, Baer 
(see \cite[Theorem 8.1]{farb})
and Mangler \cite{mangler}, that for all nonpositively curved 
closed surfaces $M$,
the inclusion $\mbox{Top}(M)\hookrightarrow G(M)$ induces a bijection on connected
components. It is also known that the inclusion
$\Diff(M)\hookrightarrow\mbox{Top}(M)$ induces an isomorphism
$\pi_0(\mbox{Diff}(M))\simeq\pi_0(\mbox{Top}(M))$ 
(see \cite[Theorem 1.13]{farb} and \cite{EE69}).
Therefore we have an isomorphism $\iota_*:\pi_0(\mbox{Diff}(M))\to\pi_0(G(M))$ induced by the inclusion 
$\Diff(M)\hookrightarrow G(M)$.
Earle and Eells \cite{EE69} showed that $\pi_i(\mbox{Diff}^0(M))$ is trivial for all
$i\geq 2$ and for all nonpositvely curved surfaces $M$. Likewise
 $\pi_i(G^0(M))=0$ if $i\geq 2$, where $G^0(M)$ is the identity
 component of $G(M)$.
 
Hence it remains to show that $\iota$ induces
an isomorphism on fundamental groups. This has to be divided into several 
cases:
\begin{itemize}
\item ($M$ has genus $\geq 2$): Earle and Eells also prove in \cite{EE69} 
that $\mbox{Diff}^0(M)$ is contractible. Since
the center of $\pi_1(M)$ is trivial, 
the components of $G(M)$ are also contractible.
Therefore $\iota:\mbox{Diff}(M)\to G(M)$ induces an isomorphism on fundamental groups.
\item ($M$ is a 2-torus): We regard $M$ as a Lie group
whose identity element we denote by $e\in M$. Let $p:G(M)\to M $ be evaluation
at $e$, i.e. $p(f)=f(e)$. Also, for $g\in M$, denote by $L_g:M\to M$
the map $L_g(g')=gg'$. Note that the identity map $id:M\to M$ factors 
through
\beq
M\xrightarrow{L} \mbox{Diff}(M)\xrightarrow{\iota} G(M)\xrightarrow{p} M,
\eeq
where $L(g)=L_g$. This induces an isomorphism
\beq
\pi_1(M,e)\xrightarrow{L_*}\pi_1(\mbox{Diff}(M),id)\xrightarrow{\iota_*}
\pi_1(G(M),id)\xrightarrow{p_*}\pi_1(M,e).
\eeq
Earle and Eells \cite{EE69} show that $\pi_1(\mbox{Diff}(M),id)$ is a free abelian group
of rank $2$. 
Then observe that each group in the sequence above is free abelian
of rank $2$, and the composition is the identity. This forces 
$\iota_*$ to be an isomorphism.
%%%%%%%%%%%%%%%%%%%%%%%%%%%%%%%%%%%%%%%%%%%%%%%%%%%
\item ($M$ is the Klein bottle):
%A standard argument shows that, for any
%$\ast\in M$, the image of the evaluation 
%map $p_*:\pi_1G(M)\to \pi_1(M)$ is central in $\pi_1(M)$.\\
There is
a smooth action $L:S^1\to \mbox{Diff}(M)$ of the circle $S^1$ on $M$.
In order to describe the smooth homomorphism $L$ we think
of $S^1=[0,2]/0\sim 2$ as obtained by identifying the end points of the
interval $[0,2]$.
and we identify $M$ with the quotient
space $S^1\times [0,1]/\sim$, where $(z,0)\sim (\bar{z},1)$ (here $S^1$
should be thought of as the unit circle in the complex plane, and $\bar{z}$ is
complex conjugation). The action $L$ is then
given by $L(s)(z,t)=(z,t+s)$, for $s\in [0,2]/0\sim 2$, where ``$+$'' here is understood to 
be addition mod $\Z$.

Let $\ast=(1,0)\in M$ be a base point. Then we have a sequence of continuous
maps
\beq
S^1\xrightarrow{L} \mbox{Diff}(M)\xrightarrow{\iota} G(M)\xrightarrow{p} M,
\eeq
where $p$ is evaluation at $\ast\in M$. Now recall that 
$\mbox{Center}(\pi_1(M))$ is an infinity cyclic group generated by the
class in $\pi_1(M,\ast)$ represented by the orbit
of $\ast\in M$ under this circle action. Then it is easy to see that the
induced map
\beq
\pi_1(S^1)
\xrightarrow{L_*}\pi_1(\mbox{Diff}(M),id)\xrightarrow{\iota_*}
\pi_1(G(M),id)\xrightarrow{p_*}\mbox{Center}(\pi_1(M,\ast))
\eeq
is an isomorphism.\\
It is shown also in \cite{EE69} that $\pi_1(\mbox{Diff}(M),id)$ is infinite cyclic.
Likewise $\pi_1(G(M),id)$ is
infinite cyclic since $\mbox{Center}(\pi_1(M,\ast))$ is infinite
cyclic. Therefore $\iota_*$ must be an isomorphism.
\end{itemize}

This concludes the proof of the theorem when $n=2$.

\textbf{Case $n=1$:} This case follows exactly as with the torus. We only
need to recall that $\mbox{Diff}(S^1)$ has only two components each of which has the
homotopy type of a circle.
%\textcolor{red}{Now suppose that that the base space $B$ is a space
%with the homotopy type of finite simplicial complex 
%$\mathcal{B}$. Let $\theta:\mathcal{B}\to B$ be a homotopy 
%equivalence and $\bar{\theta}:B\to\mathcal{B}$ be a homotopy
%inverse. It is clear that if $E\to B$ is 
%fiber homotopically trivial, then the pull-back bundle
%$\theta^*E\to\mathcal{B}$ is fiber homotopically trivial as well. 
%Hence we can apply the theorem to this bundle to conclude that 
%$\mathfrak{t}\theta^*E$ and $\mathcal{B}\times TM$ are topologically
%equivalent $TM$-bundles over $\mathcal{B}$. Therefore 
%$\mathfrak{t}E=\bar{\theta}^*\theta^*\mathfrak{t}E
%\simeq\bar{\theta}^*\mathfrak{t}(\theta^*E)$
%and $\bar{\theta}^*(\mathcal{B}\times TM)\simeq B\times TM$ are
%topologically equivalent as $TM$-bundles over $B$. This
%concludes the proof of the theorem.}
\end{proof}
\begin{rmk}
When $M$ is a closed hyperbolic 3-manifold,
Gabai shows in \cite{gabai} that the identity component of
$\mbox{Diff}(M)$
is contractible. It is also shown
in \cite{Gab97} and \cite{GMT} that the
connected components of $\mbox{Diff}(M)$ are in canonical 
bijection with the connected components of the isometry group
$\mbox{Isom}(M)$ of $M$.
This, combined with Mostow rigidity theorem, shows that
$\iota: \mbox{Diff}(M)\hookrightarrow G(M)$ is
a homotopy equivalence.

It would be interesting to know whether Theorem \ref{main1} holds 
for all smooth bundles with nonpositively
curved $3$-dimensional fiber.
\end{rmk}
\section{Vanishing of tautological classes}\label{vanishing}
In this section we prove Theorem
\ref{tangential} and Corollary \ref{main2}. 
Let $\rho:B\times M\to M$ be the
projection onto the second factor.

We now prove Theorem \ref{tangential}. 
%following result of Ferry and Weinberger:
%\begin{lemma}{\cite[Corollary 2.7]{FW91}}\label{coro2.7}
%Let $n\geq 4$ and $k$ be given. Then there is a $\delta>0$ such that if
%$X^k$ is a $k$-dimensional complex, $p:E\to X^k$ is a vector bundle with fiber
%$\R^n$, $p':E'\to X'$ is a topological $\R^n$-bundle, and $f:E\to E'$
%is a bundle map such that the restriction of $f$ to each fiber is a
%$\delta$-map, then $f^*(E')$ is isomorphic to $E$ as a topological 
%$\R^n$-bundle.
%\end{lemma}
\begin{proof}[\textbf{Proof of Theorem \ref{tangential}}]\ 

\textbf{Case $n\geq 4$}: We make use of Ferry and 
Weinberger's result (Lemma \ref{coro2.7}) again. 
%\textcolor{red}{We will assume that the base space $B$ is a finite simplicial
%complex. The general case when $B$ is a space with
%the homotopy type
%of a finite simplicial complex $\mathcal{B}$
%follows by an argument similar
%to the one used at the end of the proof of Theorem \ref{main1}.}
%i.e. by pulling-back
%the respective bundles along a homotopy equivalence
%$\theta:\mathcal{B}\to B$.
Put a nonpositively curved metric on $M$ and
consider the following diagram 
(with the same notation as in Section \ref{topotri})
\beq
\xymatrix{
& B\times TM\ar[r]\ar[d]_{id_B\times\tau_M} & \tau E\ar[d]^{t}\\
& B\times M\ar[r]^-{\Upsilon} &  E}
\eeq
where the top horizontal map is defined by
\beq
(x,v)\mapsto\left(\Upsilon(x,\tau_M(v)),
\left[\widetilde{\Upsilon}(x,\gamma_v(0)),
\widetilde{\Upsilon}(x,\gamma_v(1))\right]\right).
\eeq
Here $\gamma_v:[0,1]\to\widetilde{M}$ is the unique geodesic
with $\gamma_v(0)=\tau_M(v)$ and $\dot{\gamma}_v(0)=v$, 
and the square brackets denote an
equivalence class in $\widetilde{E}\times_{\Gamma}\widetilde{E}$.
This is a morphism of topological
$\R^n$-bundles and one easily checks (by compactness of $B$) that the 
horizontal top map is a $\delta$-map when restricted to each
fiber, for some $\delta>0$. The linear structure on the fibers
of $TM$ can be used to adjust this $\delta$ so that it is the same
as in Lemma \ref{coro2.7}.  Therefore
$\psi^*(\tau E)\simeq B\times TM$. Thus we have the following isomorphisms
of topological $\R^n$-bundles:
% For each
%$x\in B$, the following diagram commutes:
%\beq
%\xymatrixcolsep{4pc}\xymatrix{
%&TM\ar[r]^{Exp_{\, \Gamma}}\ar[d]^{\tau_M}  
%& \widetilde{M}\times_{\Gamma}\widetilde{M}\ar[r]^-
%{\ \tilde{\psi_x}\times_{\Gamma}\tilde{\psi_x}\ }   
%     &\widetilde{M_x}\times_{\Gamma}\widetilde{M_x}\ar[d]^{\ t}\\
%     &M\ar[rr]^{\psi_x}&& M_x
%}
%\eeq
%Here $Exp_{\, \Gamma}$ is the exponential map 
%$T\widetilde{M}\to\widetilde{M}\times\widetilde{M}$, sending
%$v\in T\widetilde{M}$ to $(\gamma_v(0),\gamma_v(1))$ mod $\Gamma$ 
%(note the slight difference with the map in diagram \ref{diagram1}). The other
%maps were defined in the previous section.
%
%Since $\tilde{\psi_x}$ varies continuously with $x$, the composition of the
%two horizontal maps in the top row of the diagram, gives us a map between  
%topological $\R^n$-bundles

\begin{align*}
\mathfrak{t}E & \simeq\tau E\\
			   &\simeq\Phi^*(B\times TM)\\
			   &\simeq\Phi^*\rho^*(TM),
\end{align*}
where the first isomorphism comes from Lemma 
\ref{isoasbundles}. This completes the proof of Theorem \ref{tangential}
in the case $n\geq 4$.

\textbf{Case $n=1,2$}.
Since every fiber homotopically trivial smooth $M$-bundle is smoothly
trivial (by the proof of Theorem \ref{main1}), then there exists an
isomorphism $\Lambda:E\to B\times M$ of smooth $M$-bundles over $B$.
Since the inclusion $\iota:\mbox{Diff}(M)\hookrightarrow G(M)$ is a (weak)
homotopy equivalence, it follows that 
$\Phi\circ\Lambda^{-1}:B\times M\to B\times M$ is fiberwise homotopic to
an isomorphism between smooth $M$-bundles. This yields to an isomorphism
between vector bundles over $B\times M$, namely:
\beq
(\rho\circ\Phi\circ\Lambda^{-1})^*(TM)\simeq B\times TM,
\eeq
and hence an isomorphism between vector bundles over $E$:
\beq
(\rho\circ\Phi)^*(TM)\simeq\Lambda^*(B\times TM)\simeq\mathfrak{t}E.
\eeq
This completes the proof of Theorem \ref{tangential} in the case $n=1,2$.
\end{proof}
\begin{proof}[\textbf{Proof of Corollary \ref{main2}}]
Suppose first that $n\neq 3$. Denote
by $\mathfrak{F}:BSO(n)\to B\mbox{TOP}(n)$ the 
``forget both orientation and linear structure'' map.
Then Theorem \ref{tangential} can be rephrased by saying that 
\beq
\mathfrak{F}\circ\beta \sim
\mathfrak{F}\circ\alpha\circ (\rho\circ\Phi),
\eeq
where $\beta:E\to BSO(n)$ and $\alpha:M\to BSO(n)$ are the classifying
maps for the associated vertical tangent bundle of $q:E\to B$ and
the tangent bundle of $M$, respectively.
Note that the rational Pontrjagin classes can be viewed as 
characteristic classes for topological $\R^n$-bundles 
(see \cite{kirbysiebenmann}), i.e. each universal Pontrjagin
class $p_i\in H^{4i}(BSO(n);\Q)$ has a preimage under the
induced homomorphism
\beq
\mathfrak{F}^*:H^*(\mbox{BTOP}(n);\Q)\to H^*(BSO(n);\Q).
\eeq
This implies that 
\beq
\beta^*p_i=(\rho\circ\Phi)^*\alpha^*p_i,
\eeq
for all $i$.
On the other hand, observe that the Euler class is well defined
for oriented topological $\R^n$-bundles and natural with respect to
orientation-preserving bundle maps of topological $\R^n$-bundles, 
hence
\beq
\beta^*e=\pm(\rho\circ\Phi)^*\alpha^*e,
\eeq
where $e\in H^n(BSO(n);\Q)$ is the Euler class.

Since $H^*(BSO(n);\Q)$ is a polynomial 
ring with rational coefficients
generated by the Pontrjagin classes and the Euler class, then
the last two equations imply
\beq
\beta^*c=\pm(\rho\circ\Phi)^*\alpha^*c
\eeq
for any $c\in H^*(BSO(n);\Q)$. Hence every rational tautological
class $\tau_c(E)$ of positive degree must vanish.

%%%%%%%%%%%%%%%%%%%%%%%%%%%%%%%%%%%%%%%%%%%%%%%%%%%%
Let now $n=3$. Fix a point $\ast\in S^1$.
If $M\to E\xrightarrow{q} B$ is a smooth bundle with
nonpositively curved fibers, then so is
$M\times S^1\to E\times S^1\xrightarrow{q'}B$,
where $q'(a,z)=q(a)$.  Note that we only
have to analyze one possible nonzero cohomology class, namely the
first Pontrjagin class $p_1\in H^4(BSO(n);\Q)$.  It is easy to see that 
$\sigma^*\mathfrak{t}(E\times S^1)\simeq\mathfrak{t}E\oplus\varepsilon$, where $\varepsilon$
is a trivial line bundle over $E$ and
$\sigma:E\hookrightarrow E\times S^1$ 
sends $a\in E$ to $(a,\ast)$. 
Since the dimension of $M\times S^1$ is $4$, we conclude that:
\begin{align*}
p_1(\mathfrak{t}E)&=p_1(\mathfrak{t}E\oplus\varepsilon)\\
				   &=p_1(\sigma^*\mathfrak{t}(E\times S^1))\\
				   &=\sigma^*\circ(\rho\times id_{S^1}\circ\Phi\times id_{S^1})^* p_1(M\times S^1)=0,
\end{align*}
where the last equation follows from the fact that 
$M\times S^1$ bounds a closed oriented $5$-dimensional
manifold-with-boundary.
\end{proof}
\begin{rmk}
The vanishing of the tautological classes for smooth bundles with
$3$-dimensional fibers (without curvature assumptions) is proved
by J. Ebert in \cite{ebert}.
\end{rmk}
\section{Tautological classes of nonpositively curved
bundles}\label{nonpositively}
In this section we prove Theorem \ref{add}. The strategy of proof
is the same as in Theorem \ref{tangential}. We construct
a space $\widehat{\mathfrak{t}}E$ which plays
the same role as $\widetilde{E}\times_{\Gamma}\widetilde{M}$ when
$E\to B$ is fiber homotopically trivial (see Section \ref{topotri}).
Assume that $M$ is just a closed smooth manifold 
(without curvature conditions).
For any
$\ast\in M$ define $\p_{\ast}(M)$ to be the set of homotopy classes 
(relative to the end points) of paths $\alpha:[0,1]\to M$ whose initial
point is $\ast$. For a path $\alpha$ and a neighborhood $U\subset M$ of 
$\alpha(1)$, let $(\alpha,U)$ denote the set of all paths 
$\beta:[0,1]\to M$ such that $\beta(0)=\ast$, $\beta(1)\in U$ and
there exists a path $\gamma:[0,1]\to U$ with the property that 
$\gamma(0)=\alpha(1)$, $\gamma(1)=\beta(1)$ such that the loop
$\alpha\ast\gamma\ast\beta^{-1}$ is nullhomotopic.

The sets $(\alpha,U)$ depend on the homotopy class of $\alpha$
and on the open set $U$. In fact they form a basis for the topology of 
$\p_{\ast}(M)$. With this topology $\p_{\ast}(M)$ is a simply connected
space and the map 
$\p_{\ast}(M)\to M$ sending a class $[\alpha]$ to the endpoint
$\alpha(1)$, is a covering map (see for example 
\cite[I.6]{greenberg}).

Let $\p(M)$ be the disjoint union 
$\displaystyle\bigsqcup_{\ast\in M}\p_{\ast}(M)$ topologized as follows:
For a path $\alpha:[0,1]\to M$, and neighborhoods $U$ and $V$ of 
$\alpha(0)$ and $\alpha(1)$ respectively, let 
$(\alpha, U,V)$ denote the set of all paths $\beta:[0,1]\to M$ such
that $\beta(0)\in U$, $\beta(1)\in V$ and there exist paths
$\gamma_U:[0,1]\to U$ and $\gamma_V:[0,1]\to V$ with the property that
$\gamma_U(0)=\alpha(0)$, $\gamma_U(1)=\beta(0)$, 
$\gamma_V(0)=\alpha(1)$, $\gamma_V(1)=\beta(1)$ and the loop
$\alpha\ast\gamma_V\ast\beta^{-1}\ast\gamma_U^{-1}$ is nullhomotopic.
It is straightforward that the sets $(\alpha,U,V)$ are well-defined 
and that they form a basis for the topology
of $\p(M)$. With this topology the onto map $r_0:\p(M)\to M$, sending
a class $[\alpha]$ to its initial point $\alpha(0)$ is continuous. In
fact we have the following lemma.
\begin{lemma}\label{p(M)}
If the universal cover of $M$ is homeomorphic to $\R^n$, then
\beq
\p(M)\xrightarrow{r_0}M
\eeq
is a topological $\R^n$-bundle over $M$.
\end{lemma}
\begin{proof}
Let $U\subset M$ be a simply connected neighborhood of a point $\ast\in M$.
Note that $r_0^{-1}(U)=\bigsqcup_{x\in U}\p_{x}(M)$. Let us define
a map
\beq
\rho_U:r_0^{-1}(U)\to U\times\p_{\ast}(M)
\eeq
by $\rho_U([\alpha])=(\alpha(0),[\gamma\ast\alpha])$, where 
$\gamma:[0,1]\to U$ is any path in $U$ starting at $\ast$ and ending
at $\alpha(0)$. It is clear that that this is a fiber preserving 
homeomorphism, so it defines a local trivialization for 
$r_0:\p(M)\to M$. This completes
the proof of the lemma as the universal cover
$\widetilde{M}\approx\p_{\ast}(M)$ of $M$ is homeomorphic to $\R^n$.
\end{proof}

We shall need a parametrized version of this construction. Let 
$M \to E \to B$ be a smooth fiber bundle. 
Define $\widehat{\mathfrak{t}}E:=
\displaystyle\bigsqcup_{x \in B}\p(M_{x})$. This set
can be topologized as follows: associated to
the fiber bundle $E \to B$, one can construct a concrete
principal $\Diff(M)$-bundle $Q \to B$ 
whose (concrete) fiber over $x \in B$ is the space
$\Diff(M,M_{x})$ of all diffeomorphisms between $M$
and $M_{x}$. Then a topology on $\widehat{\mathfrak{t}}E$  is
given via the bijection from $Q\times_{\Diff(M)}\p(M)$ onto 
$\widehat{\mathfrak{t}}E$ that sends 
$(\phi,\alpha) \in \Diff(M,M_{x})\times \p(M)$ to the path $\phi\circ\alpha$. Observe that the
natural projection $\widehat{\mathfrak{t}}E \to B$ is, in
fact, the 
$\p(M)$-bundle associated to $E \to B$. 

Define the projection
\beq
s_0:\widehat{\mathfrak{t}}E\to E
\eeq
by $s_0([\alpha])=\alpha(0)$. Using the fact that $\p(M)\to M$
is $\Diff(M)$-equivariant and that the natural homeomorphism
$Q\times_{\Diff(M)}M \to E$ is covered by the homeomorphism $Q\times_{\Diff(M)}\p(M)\to\widehat{\mathfrak{t}}E$, the next lemma follows obviously from Lemma \ref{p(M)}.
\begin{lemma}
If the universal cover of $M$ is homeomorphic to $\R^n$, then
$s_0:\widehat{\mathfrak{t}}E\to E$ is a topological $\R^n$-bundle.
\end{lemma}

Now assume that the bundle $M\to E\to B$ is
given a fiberwise Riemannian metric $g_x$,
$x\in B$.

Let $v\in\mathfrak{t}E$. Then $v\in T_{\ast}(M_x)$ for some $x\in B$ and
$\ast\in M_x$. Let $\gamma_v$ be the unique $g_x$-geodesic in $M_x$
such that $\gamma_v(0)=\ast$ and $\dot{\gamma}_v(0)=v$. Then define
\beq
\textsc{exp}:\mathfrak{t}E\to\widehat{\mathfrak{t}}E
\eeq
by
\beq
\textsc{exp}(v)=[\gamma_v].
\eeq

Since the restriction of $\textsc{exp}$ to each fiber is continuous and 
the Riemannian metrics $g_x$ vary continuously from fiber to fiber, 
$\textsc{exp}$ is a fiber preserving continuous map. 
Let
$s_1:\widehat{\mathfrak{t}}E\to E$ be defined by $s_1[\alpha]=\alpha(1)$.
For each $x\in B$ and $\ast\in M_x$, the map
\beq
s_1|_{s_0^{-1}(\ast)}:\p_{\ast}(M_x)\to M_x
\eeq
is a covering map. Thus we can pull-back 
both the smooth structure and the Riemannian metric $g_x$ from $M_x$
along $s_1|_{s_0^{-1}(y)}$. In this way
$\widehat{\mathfrak{t}}E\xrightarrow{s_0} E$ acquires
a fiberwise Riemannian metric (and thus a distance function on each
fiber). 
\begin{lemma}\label{expanding}
The map $\textsc{exp}:\mathfrak{t}E\to\widehat{\mathfrak{t}}E$
is an expanding map when restricted to each fiber.
\end{lemma}
\begin{proof}
Let $\widetilde{M}_x\xrightarrow{p} M_x$ be the universal cover of $M_x$
and $\tilde{\ast}\in p^{-1}(\ast)$. Consider the
following commutative diagram:
\beq
\xymatrix{
T_{\tilde{\ast}}\widetilde{M}_x\ar[d]^{p_*}\ar[r]^{\textsc{exp}}&\p_{\tilde{\ast}}(\widetilde{M}_x)\ar[d]^{\hat{p}}\ar[r]^{r_1}&\widetilde{M}_x\ar[d]^{p}\\
T_{\ast}M_x\ar[r]^{\textsc{exp}}&\p_{\ast}(M_x)\ar[r]^{r_1}&M_x
}
\eeq
where $\hat{p}$ is the lifting of $p$, i.e., 
$\hat{p}([\alpha])=[p\circ\alpha]$ and $p_*$ is the derivative of $p$ at
$\tilde{\ast}$, the map $r_1$ is defined by $r_1([\alpha])=\alpha(1)$ and 
$\textsc{exp}:T_{\tilde{\ast}}\widetilde{M}_x\to\p_{\tilde{\ast}}(\widetilde{M}_x)$ is defined with respect to the pull-back metric
$p^*g_x$ on $\widetilde{M}_x$.

Note that, when $\mathcal{P}_{\ast}(\widetilde{M}_x)$ is equipped with
the pull-back metric $(p\circ r_1)^*g_x$, the map $r_1:\p_{\tilde{\ast}}(\widetilde{M}_x)\to
\widetilde{M}_x$ becomes an isometry and $r_1\circ\textsc{exp}:
T_{\tilde{\ast}}\widetilde{M}_x\to\widetilde{M}_x$ is the usual exponential
map. This forces
$\textsc{exp}:T_{\tilde{\ast}}\widetilde{M}_x\to\p_{\tilde{\ast}}(\widetilde{M}_x)$ to be an expanding map (by Cartan-Hadamard theorem).
Since $\hat{p}$ and $p_*$ are isometries, the map 
$\textsc{exp}:T_{\ast}M_x\to\p_{\ast}(M_x)$ is an expanding map as well.
\end{proof}

\begin{proof}[\textbf{Proof of Theorem \ref{add}}]
Suppose first that $n\geq 4$.
%\textcolor{red}{We will assume that the base space $B$ is a finite simplicial
%complex. The general case when $B$ is a space with
%the homotopy type
%of a finite simplicial complex $\mathcal{B}$
%follows by an argument similar
%to the one used at the end of the proof of Theorem \ref{main1}.}
For each $x\in B$, the following diagram is commutative
\beq
\xymatrix{
TM_x\ar[d]\ar[r]^{\textsc{exp}}&\p
(M_x)\ar[r]^{\widehat{\Phi}_x}&\p
(M'_x)\ar[d]^{r_0}\\
M_x\ar[rr]^{\Phi_x}&&M'_x
}
\eeq
where $\widehat{\Phi}_x$ is the lifting of the restriction $\Phi_x$
of the fiber homotopy equivalence $\Phi$ to the fiber over $x\in B$. 
The diagram above then induces
a map between topological $\R^n$-bundles
\beq
\xymatrix{
\mathfrak{t}E\ar[d]\ar[r]&\widehat{\mathfrak{t}}E'\ar[d]\\
E\ar[r]^{\Phi} & E'
}
\eeq

Compactness of $B$ and a covering space theory argument 
(e.g. \cite[p. 409]{FW91}) show that, for all $x\in B$,
$\Phi_x$ is a $\delta$-map for some $\delta\geq 0$. Hence, by
Lemma \ref{expanding}, the map
$\mathfrak{t}E\to\widehat{\mathfrak{t}}E'$ is a fiberwise $\delta$-map.
By Lemma \ref{coro2.7},
$\Phi^*(\widehat{\mathfrak{t}}E')$ and $\mathfrak{t}E$ are isomorphic
as topological $\R^n$-bundles.
Proceeding as in the proof of Lemma
\ref{isoasbundles}, one can easily verify that 
$\widehat{\mathfrak{t}}E'$ and $\mathfrak{t}E'$ are isomorphic as
microbundles and by Kister-Mazur theorem they must be isomorphic as
topological $\R^n$-bundles. Therefore $\Phi^*(\mathfrak{t}E')$ and $\mathfrak{t}E$
are isomorphic as topological $\R^n$-bundles.

When $n=1,2$, the result follows from the fact that the fiber
homotopy equivalence $\Phi$ is homotopic to an $M$-bundle isomorphism,
provided $\Diff(M)\hookrightarrow G(M)$ is a weak homotopy equivalence
and $B$ is a finite simplicial complex. That
$\Phi$ is homotopic to an $M$-bundle isomorphism
can be easily obtained
by an inductive argument over the 
skeleta of $B$, using the vanishing
of the relative homotopy groups of the pair $(G(M),\Diff(M))$ and the
homotopy extension property.
\end{proof}
In order to prove Corollary \ref{sameclasses}
 we use the following lemma
from \cite[Theorem 5.6]{GGRW}. We will provide a different proof
of this lemma in the Appendix. 
\begin{lemma}\label{app2}
Let $\Phi:E\to E'$ be an orientation-preserving 
fiber homotopy equivalence between
oriented smooth $M^n$-bundles
over a finite simplicial complex $B$,
where $M$ is a closed smooth $n$-manifold. Then 
$\Phi^*(e(\mathfrak{t}E'))=e(\mathfrak{t}E)$, where 
$e(\mathfrak{t}E)\in H^{n}(E;\Q)$
denotes the Euler class of $\mathfrak{t}E\to E$.
\end{lemma}
\begin{proof}[\textbf{Proof of Corollary \ref{sameclasses}}]
By an obvious modification of the argument
given in the proof of Corollary \ref{main2}, one can show
that $p_i(\mathfrak{t}E)=\Phi^*p_i(\mathfrak{t}E')$ for all $i$.
This, together with Lemma \ref{app2}, implies that
%$\beta^*c=\Phi^*\beta'^*c$ for any $c\in H^*(BSO(n);\Q)$ where
%$\beta:E\to BSO(n)$ and $\beta':E'\to BSO(n)$ are the classifying
%maps for $\mathfrak{t}E$ and $\mathfrak{t}E'$, respectively. 
$\tau_c(E)=q_!\; \Phi^*\beta'^*c$ for any $c\in H^*(BSO(n);\Q)$,
where $\beta':E'\to BSO(n)$ is
the classifying map for $\mathfrak{t}E'$.
The result follows from the naturality of 
the Gysin map.
\end{proof}
\section{Tautological classes of torus and hyperbolic bundles}\label{torusbundle}

Let $G$ denote the group of affine diffeomorphisms of the $n$-dimensional
torus $\T^n$, i.e. we identify $\T^n$ with the quotient $\R^n/\Z^n$ (so
that $GL_n(\Z)$ acts naturally on $\T^n$) and define
\beq
G=\{f\in \mbox{Diff}(\T^n)|f(x)=Ax+v,\text{ where }A\in GL_n(\Z)\text{ and }v\in\T^n\}.
\eeq

The proof of Theorem \ref{torus} follows from the following three lemmas:
\begin{lemma}\label{pull-back}
Every smooth torus bundle $\T^n\to E\to B$ over a $CW$-complex $B$
is fiber homotopy equivalent
to the pull-back of the bundle $\T^n\to EG\times_G\T^n\to BG$ 
along some continuous map
$B\to BG$.
\end{lemma}
\begin{proof}
It suffices to show that the inclusion $\iota: G\hookrightarrow G(\T^n)$ 
is a weak homotopy equivalence. For in that case the induced map between
classifying spaces
 $BG\to BG(\T^n)$ is also a weak homotopy equivalence and hence we obtain a bijection 
 between the sets $[B,BG]$ and $[B,BG(\T^n)]$ of homotopy classes of 
 continuous maps from $B$ to $BG$  and to $BG(M)$ respectively.
 
 To prove that $\iota$ is a weak homotopy equivalence we note first that
 the group $G$ is isomorphic to a semidirect product $GL_n(\Z)\ltimes\T^n$.
 Thus $\pi_iG=0=\pi_iG(\T^n)$ for all $i\geq 2$. It remains to show
 that $\iota$ induces a bijection between path components 
 and an isomorphism in $\pi_1$.

Recall that $GL_n(\Z)$ acts naturally on $\T^n$ yielding a decomposition
of the identity map on $GL_n(\Z)$ as
\beq
\mbox{Out}(\pi_1(\T^n))=GL_n(\Z)\to G\xrightarrow{\iota}G(\T^n)\to
GL_n(\Z)=\mbox{Out}(\pi_1(\T^n)),
\eeq
where the first map is the action of $GL_n(\Z)$ on $\T^n$ and the last map
is the one sending a self-homotopy equivalence of $\T^n$ to its induced
(outer) automorphism in $\pi_1(\T^n)$. These two maps induce bijections at the
$\pi_0$-level. Hence $\iota:G\to G(\T^n)$ must induce a bijection at the
$\pi_0$-level as well.

Let now $0\in\T^n$ denote the identity element of 
(the abelian group) $\T^n$ and for each
$v\in\T^n$ let $L_v:\T^n\to \T^n$ be the map that sends $x\in\T^n$ to $x+v$.
The identity map $id:\T^n\to\T^n$ decomposes as
\beq
\T^n\xrightarrow{L} G\xrightarrow{\iota}G(\T^n)\xrightarrow{p}\T^n,
\eeq
where $L(v)=L_v$ and $p:G(\T^n)\to\T^n$ is the evaluation at $0\in\T^n$. Thus 
we have an isomorphism
\beq
\pi_1(\T^n,0)\xrightarrow{L_*}\pi_1(G,id)\xrightarrow{\iota_*}
\pi_1(G(\T^n),id)\xrightarrow{p_*}\pi_1(\T^n,0),
\eeq
where each group in the sequence is a free abelian group of rank $n$.
Therefore $\iota_*$
must be an isomorphism. This completes the proof of the lemma. 
\end{proof}
\begin{lemma}\label{affineflat}
Let  $\T^n\to E\to B$ be a smooth torus bundle over a compact space
$B$ and let $E'\xrightarrow{p}B$ be the pull-back bundle of 
$\T^n\to EG\times_G\T^n\to BG$ along a continuous map $B\to BG$.
Suppose that
there is a fiber homotopy equivalence $\Psi:E'\to E$. Then 
$\Psi^*\mathfrak{t}E$ and $\mathfrak{t}E'$ are isomorphic as topological
$\R^n$ bundles over $E$.
\end{lemma}
\begin{proof}
First note that the bundle $\T^n\to E'\xrightarrow{p} B$
can be given a fiberwise affine flat connection in the following way: 
let $V\subset B$ be an open subset such that there is a fiber preserving diffeomorphism 
$\phi_V:V\times\T^n\to p^{-1}(V)$. Identify $\T^n$
with $\R^n/\Z^n$ and endow it with a flat Riemannian metric
$g^F$ induced by the Euclidean metric on $\R^n$. Denote by $\nabla^F$
the Levi-Civita connection of $g^F$. Then
we put a connection $\nabla_x$ on 
$p^{-1}(x)=\T^n_x$ by ``pushing forward'' $\nabla^F$ via $\phi_V$, i.e.
\beq
\nabla_x=
\left(\left. \phi_V\right|_{\{x\}\times\T^n}\right)_*\nabla^F
\ \ \ \text{  for  } x\in V.
\eeq
Now cover $B$ with finitely many open sets $V_1,\ldots, V_m$,
as above and define affine 
connections in a similar way by pushing forward $\nabla^F$ via
diffeomorphisms 
$\phi_{V_{i}}:V_{i}\times\T^n\to p^{-1}(V_{i})$.
Since the structure group of the bundle
$E'\xrightarrow{p}B$ is the group $G$ of affine
diffeomorphisms of the torus, these connections can be glued together to
give rise to a desired fiberwise affine flat connection $\nabla_x$, 
$x\in B$.

Having a fiberwise affine connection allows us to define an 
exponential map. Let $v\in\mathfrak{t}E'$. Then $v\in T_{\ast}(\T^n_x)$
for some $x\in B$ and $\ast\in\T^n_x$. Let $\beta_v$ be the unique
geodesic (with respect to the connection $\nabla_x$), such that
$\beta_v(0)=\ast$ and $\dot{\beta}_v(0)=v$. We define a map 
$\mbox{exp}^{\nabla}:\mathfrak{t}E'\to\widehat{\mathfrak{t}}E'$
between $\R^n$-bundles by
\beq
\mbox{exp}^{\nabla}(v)=[\beta_v],
\eeq
where $\widehat{\mathfrak{t}}E'$ was defined in the previous section. This
map is continuous because the affine connections $\nabla_x$ vary 
continuously with $x$.

If we ``fiberwise lift'' the fiber homotopy equivalence $\Psi:E'\to E$
to a map
$\widehat{\Psi}:\widehat{\mathfrak{t}}E'\to\widehat{\mathfrak{t}}E$
(cf. proof of Theorem \ref{add}),
we obtain a map between $\R^n$-bundles
\beq
\mathfrak{t}E'\xrightarrow{\mbox{exp}^{\nabla}}\widehat{\mathfrak{t}}E'
\xrightarrow{\widehat{\Psi}}\widehat{\mathfrak{t}}E
\eeq
covering $\Psi:E'\to E$. 
Using an argument similar to that of Lemma \ref{isoasbundles}, it is not
hard to see that $\widehat{\mathfrak{t}}E$ is isomorphic
to $\mathfrak{t}E$ as 
$\R^n$-bundles. Thus by Lemma \ref{coro2.7}, it only remains to show
that $\widehat{\Psi}\circ\mbox{exp}^{\nabla}$ is a $\delta$-map (for a suitable
$\delta\geq 0$) when restricted to each fiber.

Let $g_x$ be the push-forward metric of the flat Riemannian metric on $\T^n$ 
along the restriction
of $\phi_V$ to $\{x\}\times\T^n$. By construction, the 
Levi-Civita 
connection of $g_x$ is $\nabla_x$. Thus, for all $x\in V$ and
$\ast\in\T^n_x$, we have
\beq
d_{g_x}(\mbox{exp}^{\nabla}(v),\mbox{exp}^{\nabla}(w))
\geq d_{\hat{g}_x}(v,w), 
\text{  for all  } v,w\in T_{\ast}\T^n_x,
\eeq
where $d_{g_x}$ and $d_{\hat{g}_x}$ are the distance functions on
$\T^n_x$ and  $T_{\ast}\T^n_x$ respectively, induced by the Riemannian metric $g_x$.

Now let $\{h_x\}_{x\in B}$ be a continuously varying family of Riemannian metrics
on the bundle $\T^n\to E'\xrightarrow{q} B$. It is not hard to see that
there exists a number $K_V\geq 1$ that depends on the open set $V$, such
that for all $x\in V$,
\begin{align}\label{ineq}
K_V d_{h_x}(\mbox{exp}^{\nabla}(v),\mbox{exp}^{\nabla}(w))\nonumber
&\geq d_{g_x}(\mbox{exp}^{\nabla}(v),\mbox{exp}^{\nabla}(w))\\
&\geq d_{\hat{g}_x}(v,w)\\
&\geq\frac{1}{K_V} d_{\hat{h}_x}(v,w)\nonumber.
\end{align}

Then we have
that for all $x\in B$
\beq
d_{h_x}(\mbox{exp}^{\nabla}(v),\mbox{exp}^{\nabla}(w))\geq\frac{1}{K^2} d_{\hat{h}_x}(v,w),
\eeq
where $K=\max\{K_{V_i}|i\in\{1,\ldots,m\}\}$.

By using this inequality, 
the compactness of $B$,
the fact that $\widehat{\Psi}$ is a $\delta'$-map when restricted to
each fiber
(see \cite[p. 409]{FW91}), and the linear structure on the fibers of
$\mathfrak{t}E'$, it is straightforward
that the composite $\widehat{\Psi}\circ\mbox{exp}^{\nabla}$
is a $\delta$-map (for a $\delta\geq 0$ as in the statement of Lemma
\ref{coro2.7}) when restricted to each fiber.
This completes the proof of Lemma \ref{affineflat}.
\end{proof}

\begin{lemma}\label{discrete}
Suppose that $B$ is a smooth manifold.
Then, for any continuous map $f:B\to BG$, the rational Pontrjagin
classes of the associated vertical tangent bundle
$\eta:\mathfrak{t}(f^*(EG\times_G\T^n))\to f^*(EG\times_G\T^n)$ vanish.
\end{lemma}
\begin{proof}
It suffices to prove that the structure group of 
$r:\mathfrak{t}(EG\times_G\T^n)\to EG\times_G\T^n$ (and hence
of $\eta:\mathfrak{t}(f^*(EG\times_G\T^n))\to f^*(EG\times_G\T^n)$) can be reduced
to the discrete group $GL_n(\Z)$. For if that is the case then 
$\eta$
is a flat vector bundle over a smooth manifold and so its 
rational Pontrjagin classes vanish. (However the rational Euler
class does not necessarily vanish for flat bundles. See 
\cite[Appendix C, p. 308, 312]{milnor}).

To prove that the structure group can be reduced to a discrete
group we first note that
there is a left action of $G$ on the tangent bundle $T\T^n$ of the torus
given by the derivative of the $G$-action on $\T^n$. Consider the vector bundle
\beq
\R^n\to EG\times_G T\T^n\xrightarrow{\hat{r}}EG\times_G\T^n,
\eeq
where $\hat{r}$ is defined by
\beq
\hat{r}[(e,V)]=[(e,\tau(V))],
\eeq
and $\tau:T\T^n\to \T^n$ is the tangent bundle projection. Now, since
$GL_n(\Z)$ sits in an exact sequence of groups
\beq
1\to \T^n\to G\to GL_n(\Z)\to 1,
\eeq
the action of $G$ on $EG\times\T^n$ induces a free action of $GL_n(\Z)$ on
$EG\times_{\T^n}\T^n$ whose orbit space is $EG\times_G\T^n$. On the other
hand, we have the usual action of $GL_n(\Z)$ on $\R^n$. Thus we can form
a vector bundle
\beq
\R^n\to (EG\times_{\T^n}\T^n)\times_{GL_n(\Z)}\R^n\to EG\times_G\T^n.
\eeq
We claim that $\mathfrak{t}(EG\times_G\T^n)$ and $(EG\times_{\T^n}\T^n)\times_{GL_n(\Z)}\R^n$ are isomorphic vector bundles
over $EG\times_G\T^n$. This claim proves the lemma because
$GL_n(\Z)$ is the structure group of the latter vector bundle.

To prove the claim we observe that $\mathfrak{t}(EG\times_G\T^n)$ 
is canonically
isomorphic to $EG\times_{G}T\T^n$. On the other hand,
the derivative of left 
translation on $\T^n$ gives rise to an isomorphism 
$\T^n\times\R^n\simeq T\T^n$ and hence to continuous map
$(EG\times_{\T^n}\T^n)\times_{GL_n(\Z)}\R^n\to EG\times_{G}T\T^n$ which
covers the identity map on $EG\times_G\T^n$ and is a linear isomorphism
on each fiber. Hence $(EG\times_{\T^n}\T^n)\times_{GL_n(\Z)}\R^n$ and
$EG\times_{G}T\T^n$ are isomorphic vector bundles.
\end{proof}
\begin{proof}[\textbf{Proof of Theorem \ref{torus}}]
Using Lemma \ref{pull-back}, we can find a continuous map
$f:B\to BG$ and a fiber homotopy equivalence
$\Psi:f^*(EG\times_G\T^n)\to E$ between torus bundles 
over the smooth manifold $B$.

By Lemma \ref{affineflat}, $\Psi^*(\mathfrak{t}E)$ and
$\mathfrak{t}(f^*(EG\times_G\T^n))$ are isomorphic topological 
$\R^n$-bundles. Thus these two bundles have the same rational Pontrjagin 
classes. But they vanish on the latter by Lemma \ref{discrete}.
%Note that $p_i(\mathfrak{t}(f^*(EG\times_G\T^n)))=p_i(f^*\mathfrak{t}(EG\times_G\T^n))$ because $\widehat{f}:E'\to EG\times_G\T^n$ is a diffeo
%fiberwise and so 
%$\widehat{f}^*(\mathfrak{t}EG\times_G\T^n)\simeq
%\mathfrak{t}f^*(EG\times_G\T^n)$ as vector bundles.
%Since $E'=f^*(E_{\; \T^n})$, there is a continuous map 
%$\hat{f}:E'\to E_{\; \T^n}$ covering $f:B\to BG$, such that $\hat{f}$ is
%a diffeomorphism when restricted to each fiber. Thus $\mathfrak{t}E'$ and
%$\hat{f}^*\mathfrak{t}(E_{\; \T^n})$ are isomorphic vector bundles over
%$E'$.
\end{proof}

We conclude this section by proving
Theorem \ref{hyperbolic}. Let $M$ be a (real, complex or
quaternionic)
hyperbolic manifold. By Mostow rigidity 
theorem \cite{mostow}, there is a group isomorphism
\beq
\mbox{Out}(\pi_1M)\to\mbox{Isom}(M)
\eeq
such that
\beq
\mbox{Out}(\pi_1M)\to\mbox{Isom}(M)
\xrightarrow{\iota}G(M)\to\mbox{Out}(\pi_1M)
\eeq
is the identity map.
%By Mostow Rigidity Theorem, the group 
%$\mbox{Out}(\pi_1M)$ of outer automorphisms of $\pi_1M$ acts on $M$, and
%the space $E\mbox{Out}(\pi_1M)\times_{\mbox{Out}(\pi_1M)}M$ is defined.
\begin{lemma}\label{homotopytype}
Every smooth bundle $M\to E\to B$ as in the statement of 
Theorem \ref{hyperbolic} is fiber homotopy equivalent to the pull-back
of 
$M\to E\mbox{Out}(\pi_1M)\times_{\mbox{Out}(\pi_1M)}M\to B\mbox{Out}(\pi_1M)$ along some continuous map 
$B\to B\mbox{Out}(\pi_1M)$.
\end{lemma}
\begin{proof}
The inclusion map 
$\iota:\mbox{Isom}(M)\to G(M)$ must induce a bijection between components.
Also, $\pi_iG(M)=0$ for all $i>0$ since $\pi_1M$ is centerless
(see Remark \ref{centerless} below). This
implies that the induced map
$B\mbox{Out}(\pi_1M)\to BG(M)$ is a weak homotopy equivalence and the
result follows.
\end{proof}
\begin{rmk}\label{centerless}
That $\pi_1M^n$ ($n\geq 2$) is centerless for a closed negatively
curved manifold $M$ is implicitly proven in 
\cite[Section 9]{EO}. For the reader's convenience, here is an
outline of the argument. Let $\widetilde{M}$ denote a universal
cover of $M$ and identify $\pi_1M$ with the group of deck
transformations of $\widetilde{M}\to M$. Each element 
$f\in\pi_1M-id$ leaves invariant a unique geodesic line $L_f$
in $\widetilde{M}$ (cf. \cite[7.1]{BGS}).
Suppose $\mbox{Center}(\pi_1M)\neq 1$ and pick an element 
$g\in \mbox{Center}(\pi_1M)$ different from the identity. Let
$f\in\pi_1M$. Since $fg=gf$, $f$ must leave $L_g$ invariant. 
Hence the deck transformation action restricts to a free and 
properly discontinuous action of $\pi_1M$ on $L_g$. Therefore
the orbit map $p:L_g\to L_g/\pi_1M$ is the universal covering
space for the 1-dimensional manifold $L_g/\pi_1M$. Hence both
$M^n$ and $L_g/\pi_1M$ are Eilenberg-MacLane spaces 
$K(\pi_1M,1)$. Therefore the closed manifold $M$ is homotopy equivalent to a closed 1-dimensional manifold. 
Hence $M$ itself must
be 1-dimensional, which is a 
contradiction.
\end{rmk}
\begin{proof}[\textbf{Proof of Theorem \ref{hyperbolic}}]
Let $f:B\to B\mbox{Out}(\pi_1M)$ be a continuous map such that there is a fiber
homotopy equivalence $E'\to E$ where
$E'=f^*(E\mbox{Out}(\pi_1M)\times_{\mbox{Out}(\pi_1M)}M)$. Since
$\mbox{Out}(\pi_1M)$ acts on $M$ by isometries, it is easy to equip the associated bundle
$M\to E\mbox{Out}(\pi_1M)\times_{\mbox{Out}(\pi_1M)}M\to B\mbox{Out}(\pi_1M)$ over
$B\mbox{Out}(\pi_1M)$ with a
nonpositively curved fiberwise Riemannian metric. Thus, by 
Corollary \ref{sameclasses}, $\tau_c(E')=\tau_c(E)$. But since
$\mbox{Out}(\pi_1M)$  is a finite group \cite{Borel}, the rational
cohomology of $B\mbox{Out}(\pi_1M)$ vanishes in positive degrees and so
the rational
tautological classes of the bundle $E'$ must be zero in positive degrees.
\end{proof}
\begin{rmk}
Theorem \ref{hyperbolic} also follows as particular case of 
Theorem \ref{app1} in the Appendix.
\end{rmk}
%%%%%%%%%%%%%%%%%%%%%%%%%%%%%%%%%%%%%%%%%%%%%%%%%%%%%%%%%%%%%%%%%%
%%%%%%%%%%%%%%%%%%%%%%%%%%%%%%%%%%%%%%%%%%%%%%%%%%%%%%%%%%%%%%%%%%
\section{Appendix}\label{appendix}
We begin with the proof of Theorem \ref{app1}.
\begin{proof}[\textbf{Proof of Theorem \ref{app1}}]
Let $\Diff_0(M)\subset\Diff(M)$ be the subgroup of $\Diff(M)$
consisting
of all self-diffeomorphisms of $M$ which are \textit{homotopic}
to the identity (note that $\Diff_0(M)$ is different 
from the identity component of $\Diff(M)$. In fact the former
contains the latter).

Consider the following pull-back diagram
\beq
\xymatrix{
\widetilde{B}\ar[d]_{\sigma}\ar[r]& B\Diff_0(M)\ar[d]\\
B\ar[r] & B\Diff(M)
}
\eeq
where $B\to B\Diff(M)$ is the classifying map for the bundle
$q:E\to B$, and $B\Diff_0(M)\to B\Diff(M)$ is the map induced by
the inclusion $\Diff_0(M)\hookrightarrow\Diff(M)$. We claim that
$\widetilde{B}\xrightarrow{\sigma}B$ is a finite sheeted 
cover of $B$. To see this, it suffices to
prove that the fiber $\Diff(M)/\Diff_0(M)$ of 
$B\Diff_0(M)\to B\Diff(M)$ is finite. But this follows
because there is a one-to-one map 
$\Diff(M)/\Diff_0(M)\to\pi_0(G(M))\simeq\mbox{Out}(\pi_1M)$, and 
$\mbox{Out}(\pi_1M)$ is finite by hypothesis.

We now consider the pull-back bundle $M\to\sigma^*E\to\widetilde{B}$ of
$M\to E\xrightarrow{q} B$ along $\sigma$. This bundle is fiber
homotopically trivial since the 
map $\widetilde{B}\to B\Diff_0(M)\to B\Diff(M)\to BG(M)\approx B\pi_0G(M)$ 
is null-homotopic; which follows from the fact that the composition
$\Diff_0(M)\hookrightarrow\Diff(M)\hookrightarrow G(M)\to\pi_0G(M)$
is a constant map, and $BG(M)\to B\pi_0G(M)$ is a weak homotopy equivalence (because $M$ is aspherical and $\pi_1M$ is centerless 
by assumption).

Hence, by Corollary \ref{main2}, 
\beq
\tau_c(\sigma^*E)=0
\eeq
for $c\in H^i(BSO(n);\Q)$ and $i>n$.

Note also that by naturality of the tautological classes
\beq
\tau_c(\sigma^*E)=\sigma^*\tau_c(E).
\eeq
Thus, it suffices to prove that 
$\sigma^*:H^*(B;\Q)\to H^*(\widetilde{B};\Q)$ is monic. But this follows
easily 
from the fact that $\sigma:\widetilde{B}\to B$ is finite sheeted covering
space, hence there is a transfer map 
$H^*(\widetilde{B};\Q)\to H^*(B;\Q)$ whose composition with
$\sigma^*$ is multiplication by the number of sheets.
\end{proof}

We now prove Theorem \ref{app3}. Recall that a closed manifold $M$
satisfies the Strong Borel Conjecture (SBC) if, for 
all $k\geq 0$, every self-homotopy equivalence of pairs
$(M\times D^k,M\times S^{k-1})\to (M\times D^k,M\times S^{k-1})$
which is a homeomorphism when restricted to the 
boundary $M\times S^{k-1}$ is
homotopic (relative to the boundary) to a homeomorphism.
\begin{proof}[\textbf{Proof of Theorem \ref{app3}}]
%\textcolor{red}{We will assume that 
%the base space $B$ is a finite simplicial
%complex. The general case when $B$ is a space with 
%the homotopy type
%of a finite simplicial complex $\mathcal{B}$
%follows by an argument similar
%to the one used at the end of the proof of Theorem \ref{main1}.}
The first step is proving that the fiber homotopy equivalence $\Phi$ is
homotopic to a homeomorphism. Let 
$B^0\subset B^1\subset\cdots\subset B^k=B$ be the skeletal filtration of
$B$. 

First consider a $0$-simplex $v\in B$. The restriction 
$\Phi|_{p^{-1}(v)}:p^{-1}(v)\to q^{-1}(v)$ is a homotopy equivalence.
Thus, by SBC, it is homotopic to a homeomorphism, that is, there is a
homotopy $h_t^v:p^{-1}(v)\to q^{-1}(v)$ such that 
$h_0^v=\Phi|_{p^{-1}(v)}$ and $h_1^v$ is a homeomorphism. 
By taking the corresponding homeomorphism on each $0$-simplex 
of $B$, we obtain a homeomorphism $h^0:p^{-1}(B^0)\to q^{-1}(B^0)$, which
is compatible with the bundle projections and homotopic to the restriction
of $\Phi$ to $p^{-1}(B^0)$.

%Let us define a map
%\beq
%p^{-1}(B^0)\times [0,1]\to q^{-1}(B^0)\times [0,1],
%\eeq
%by $(y,t)\mapsto (h^{p(y)}_t(y),t)$.
%This map coincides with $\Phi$ on
%$p^{-1}(B^0)\times \{0\}$ and with the homeomorphism $h^0$ on
%$p^{-1}(B^0)\times \{1\}$.

We now want to extend this homotopy over the $1$-skeleton. Let
$\Delta^1\subset B$ be a $1$-simplex in $B$ with vertices $v,w\in B$. Define
subspaces $E^1\subset E\times [0,1]$ and $E'^1\subset E'\times [0,1]$ by
\beq
E^1=p^{-1}(v)\times [0,1]\cup p^{-1}(\Delta^1)\times\{0\}
\cup p^{-1}(w)\times [0,1]
\eeq
and
\beq
E'^1=q^{-1}(v)\times [0,1]\cup q^{-1}(\Delta^1)\times\{0\}
\cup q^{-1}(w)\times [0,1],
\eeq
and a homotopy equivalence $\psi:E^1\to E'^1$ by
\beq
\psi(y,t)=
\begin{cases}
(h_t^v(y),t)&\text{  if    } (y,t)\in p^{-1}(v)\times [0,1]\\
(h_t^w(y),t)&\text{  if    } (y,t)\in p^{-1}(w)\times [0,1]\\
(\Phi(y),0) &\text{  if    } y\in p^{-1}(\Delta^1).
\end{cases}
\eeq

The figure below illustrates the construction.

\begin{figure}[!h]
\vspace*{0cm}
    \begin{center}
    \includegraphics[scale=0.7]{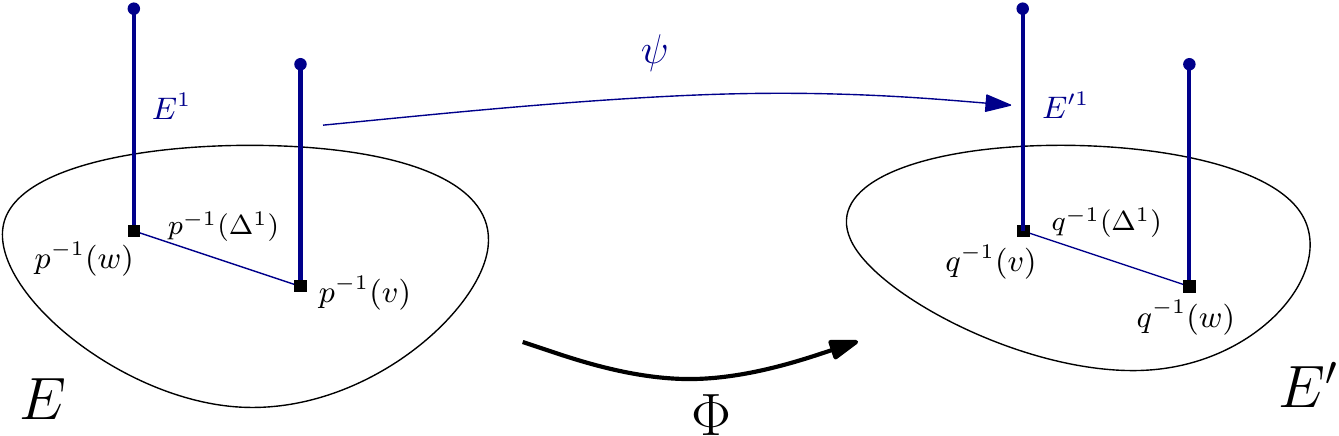}
    \caption{\small Construction of a homeomorphism homotopic
    to $\Phi$.}
    \end{center}
 \end{figure}
 
Note that both $E^1$ and $E'^1$ are homeomorphic to $M\times D^1$. Thus
map $\psi$ can be thought of as a homotopy equivalence of pairs
$(M\times D^1,M\times S^0)\to (M\times D^1,M\times S^0)$ which is
the homeomorphism 
$h^{0}$ when restricted to the
boundary 
$p^{-1}(v)\times\{1\}\cup p^{-1}(w)\times\{1\}\approx M\times S^{0}$. Thus, by SBC,  we can find a homotopy
\beq
h^{\Delta^1}_t:p^{-1}(\Delta^1)\to q^{-1}(\Delta^1)
\eeq
which extends $h_t^v$ and $h_t^w$ and such that $h^{\Delta^1}_0=\Phi|_{p^{-1}(\Delta^1)}$ and $h^{\Delta^1}_1$ 
is a homeomorphism. Repeating this process for each $1$-simplex of $B$ we obtain a
homeomorphism
$h^1:p^{-1}(B^1)\to q^{-1}(B^1)$ homotopic to the 
restriction of $\Phi$ to $p^{-1}(B^1)$. 

Continuing this way on each skeleton we obtain the
desired homeomorphism $E\to E'$ homotopic to $\Phi$.
 
The second and final step to prove the theorem is showing that
$\Phi^*\mathfrak{t}E'$ and $\mathfrak{t}E$ are stably isomorphic as
topological $\R^n$-bundles. For this, embed $B$ into the Euclidean
space $\R^N$ for some $N$ large enough. 
Take a compact regular neighborhood $\mathcal{B}$ 
of $B$ in $\R^N$ whose tangent bundle $T\mathcal{B}$ is trivial and such that there is a retraction $r:\mathcal{B}\to B$ of
$\mathcal{B}$ onto $B$.
Let $p:\mathcal{E}\to\mathcal{B}$ and
$p':\mathcal{E'}\to\mathcal{B}$
denote the pull-back bundle of $q:E\to B$  and $q':E'\to B$ along $r$
respectively. Observe that there are natural fiber homotopy equivalence
$\varphi:\mathcal{E}\to\mathcal{E'}$ and bundle maps 
$\hat{\iota}:E\to\mathcal{E}$, $\hat{\iota}':E'\to\mathcal{E}'$ covering
the inclusion $B\hookrightarrow\mathcal{B}$, such that the following 
diagram commutes
\beq
\xymatrix{
E\ar[r]^-{\hat{\iota}}\ar[d]_-{\Phi} & \mathcal{E}\ar[d]^{\varphi}\\
E'\ar[r]^-{\hat{\iota}'}			  & \mathcal{E}'
}
\eeq

Hence, by the first step,
$\varphi^*T\mathcal{E'}$ and $T\mathcal{E}$ are isomorphic as topological
$\R^n$-bundles. Therefore 
$\varphi^*\mathfrak{t}\mathcal{E'}\oplus q^*T\mathcal{B}$ and
$\mathfrak{t}\mathcal{E}\oplus q^*T\mathcal{B}$ are isomorphic as
topological $\R^n$-bundles. But since $T\mathcal{B}$ is a trivial vector bundle of rank, say $r$, then $\varphi^*\mathfrak{t}\mathcal{E'}\oplus\varepsilon_{\mathcal{E}}^r$ and
$\mathfrak{t}\mathcal{E}\oplus\varepsilon_{\mathcal{E}}^r$, where
$\varepsilon_{\mathcal{E}}^r$ is the trivial
bundle of rank $r$ over $\mathcal{E}$, are isomorphic as topological 
$\R^n$-bundles.

Finally, the topological $\R^n$-bundle isomorphism above, together with the 
isomorphisms $\hat{\iota}^*\mathfrak{\mathcal{E}}\simeq\mathfrak{t}E$ and
$\hat{\iota}^*\varphi^*\mathfrak{t}\mathcal{E'}\simeq
\Phi^*\hat{\iota}'^*\mathfrak{t}\mathcal{E}'\simeq
\Phi^*\mathfrak{t}E'$, imply that
$\mathfrak{t}E\oplus\varepsilon_E^r$ and
$\Phi^*\mathfrak{t}E'\oplus\varepsilon_E^r$ are isomorphic as topological
$\R^n$-bundles. This completes the proof of the theorem.
\end{proof}
%%%%%%%%%%%%%%%%%%%%%%%%%%%%%%%%%%%
\begin{proof}[\textbf{Proof of Corollary \ref{app4}}]
Recall that 
\beq
H^*(BSO(n);\Q)\simeq
\begin{cases}
\Q[p_1,\ldots,p_k] & \text{  if  } n=2k+1,\\
\Q[p_1,\ldots,p_k,e]/(e^2-p_k)  & \text{  if  } n=2k
\end{cases}
\eeq 
Then the corollary follows immediately from Lemma
\ref{app2} and Theorem \ref{app3}.
\end{proof}
%%%%%%%%%%%%%%%%%%%%%%%%%%%%%%%%%%%%%%%%%%%%%%%%%%%%%%%%%%%%%%%%%%
We conclude this article proving Lemma \ref{app2} which appears in
Section \ref{nonpositively}.
\begin{customlemma}{\ref{app2}}
Let $\Phi:E\to E'$ be an orientation-preserving 
fiber homotopy equivalence between
oriented smooth $M^n$-bundles
over a finite simplicial complex $B$,
where $M$ is a closed smooth $n$-manifold. Then 
$\Phi^*(e(\mathfrak{t}E'))=e(\mathfrak{t}E)$, where 
$e(\mathfrak{t}E)\in H^{n}(E;\Q)$
denotes the Euler class of $\mathfrak{t}E\to E$.
\end{customlemma}
\begin{proof}[\textbf{Proof of Lemma \ref{app2}}]
Assume first that the base space $B$ is a 
$k$-dimensional oriented
manifold (possibly with boundary). In that case
both $E$ and $E'$ are compact $(n+k)$-manifolds (possibly with boundary).

Let $E\times_B E\to B$ be the pull-back of 
the product bundle $E\times E\to B\times B$ along the diagonal map 
$B\to B\times B$ and denote by $\delta:E\to E\times_B E$ 
the fiberwise diagonal map. Define similarly the bundle
$E'\times_BE'\to B$ and the diagonal map $\delta':E'\to E'\times_BE'$.
Observe that the Euler class of $\mathfrak{t}E$ (resp. $\mathfrak{t}E'$)
can be computed by the formula
\beq
e(\mathfrak{t}E)=\delta^*\delta_!(1) \text{   (resp.   }
e(\mathfrak{t}E')=\delta'^*\delta'_!(1)).
\eeq
Here $\delta_!$ is defined via Lefschetz duality as the map making the
following diagram commutative:
\beq
\xymatrix{
H^0(E)\ar[r]^-{\delta_!}\ar[d]_{\eta_{E}\cap} & H^n(E\times_BE)\ar[d]^{\eta_{E\times_BE}\cap}\\
H_{n+k}(E,\partial(E))\ar[r]^-{\delta_*} 	  & H_{n+k}(E\times_BE,\partial(E\times_BE))
}
\eeq
where $\eta_E\in H_{n+k}(E,\partial E)$ denotes the orientation 
class of $E$.
We also have the following commutative diagram of pairs:
\beq
\xymatrix{
(E,\partial E)\ar[r]^-{\delta}\ar[d]_-{\Phi} & (E\times_BE,\partial (E\times_BE))\ar[d]^{\Phi\times_B\Phi}\\
(E',\partial E')\ar[r]^-{\delta'}		  & (E'\times_BE',\partial (E'\times_BE'))
}
\eeq

We now make use of the two diagrams above to compute:
\begin{align*}
\Phi^*e(\mathfrak{t}E')
			&=\Phi^*\delta'^*\delta'_!(1)\\
			&=\delta^*(\Phi\times_B\Phi)^*\delta'_!(1)\\
			&=\delta^*(\Phi\times_B\Phi)^*\overline{\delta'_*\eta_{E'}}\\
			&=\delta^*(\Phi\times_B\Phi)^*\overline{\delta'_*\Phi_*\eta_{E}}\\
			&=\delta^*(\Phi\times_B\Phi)^*\overline{(\Phi\times_B\Phi)_*\delta_*\eta_E}\\
			&=\delta^*\overline{\delta_*\eta_E}\\
			&=\delta^*\delta_!(1)=e(\mathfrak{t}E).
\end{align*}
Here, a bar on a homology
class denotes its Poincar\'e - Lefschetz dual and the sixth
equation follows from the naturality of the cap product and 
the fact that $\Phi\times_B\Phi$ is a homotopy equivalence. This completes
the proof of the theorem when $B$ is manifold with boundary.

Now suppose that $B$ is a finite dimensional simplicial complex. 
Embed $B$ in $\R^N$, for some $N$ sufficiently large. Let $W$ be an oriented
regular neighborhood of $B$ in $\R^N$ such that $W$ is a 
manifold-with-boundary. Denote by $r:W\to B$ a retraction of $W$ onto $B$. 
We can pull-back the bundles $E\to B$ and $E'\to B$ 
along the retraction $r:W\to B$ to obtain $M$-bundles 
$\mathcal{E}\to W$ and $\mathcal{E'}\to W$ respectively.
These two bundles are fiber homotopy equivalent via the restriction
of $id_W\times\Phi:W\times E\to W\times E'$ to $\mathcal{E}$. 
Hence, if we denote this restriction map by $\Psi$, we have
\beq
\Psi^*(e(\mathfrak{t}\mathcal{E}'))=e(\mathfrak{t}\mathcal{E}).
\eeq 
The theorem then follows by noticing that the following diagram
commutes
\beq
\xymatrix{
E\ar[r]^{\hat{\iota}}\ar[d]_{\Phi}&\mathcal{E}\ar[d]^{\Psi}\\
E'\ar[r]^{\hat{\iota}'} & \mathcal{E}',
}
\eeq
where $\hat{\iota}:E\to\mathcal{E}$ 
(resp. $\hat{\iota}':E'\to\mathcal{E}'$) is the bundle
map covering the inclusion
$B\hookrightarrow W$, and that 
$\hat{\iota}^*\mathfrak{t}\mathcal{E}\simeq\mathfrak{t}E$ 
(resp. $\hat{\iota}'^*\mathfrak{t}\mathcal{E}'\simeq\mathfrak{t}E'$).
%\textcolor{red}{Finally, if
%the base space $B$ is a space with 
%the homotopy type
%of a finite simplicial complex $\mathcal{B}$, the result
%follows by an argument similar
%to the one used at the end of the proof of Theorem \ref{main1}.}
\end{proof}
%%%%%%%%%%%%%%%%%%%%%%%%%%%%%%%%%%%%%%%%%%%%%%%%%%%%%%%%%%%%%%%%%%
\bibliographystyle{alpha} % style for bibliography
 {\footnotesize
 \bibliography{biblio}}     % biblio.bib is the name of our database
 
 \Addresses
\end{document}